  \pdfoutput=1
  \documentclass[11pt, a4paper, notitlepage]{article}
  \usepackage[left=4.1cm,right=4.1cm,top=3.3cm,bottom=3.3cm]{geometry}
  \usepackage{amsthm}
\usepackage[latin1]{inputenc}
\usepackage[english, british]{babel}
\usepackage{calc,amssymb,mathtools}
\mathtoolsset{mathic}
\usepackage{units,tabularx,paralist}
  \usepackage{hyperref} \hypersetup{breaklinks=true,colorlinks=true,linkcolor=black,anchorcolor=black,citecolor=black,filecolor=black,menucolor=black,pagecolor=black,urlcolor=blue}
\usepackage[alphabetic,non-sorted-cites]{amsrefs}
\usepackage[all]{xy}
\usepackage[nameinlink]{cleveref}

\title{Symmetric representation rings are \(\lambda\)-rings}
\author{Marcus Zibrowius}

\chardef\bslash=`\\ 

\swapnumbers
\theoremstyle{plain} 
\newtheorem{thm}{Theorem}[section]\crefname{thm}{theorem}{theorems}
\crefname{mdthm}{theorem}{theorems}
\newtheorem*{thm*}{Theorem}

\newtheorem{cor}[thm]{Corollary}\crefname{cor}{corollary}{corollaries}
\crefname{mdcor}{corollary}{corollaries}
\newtheorem*{cor*}{Corollary}

\newtheorem{lem}[thm]{Lemma}\crefname{lem}{lemma}{lemmas}
\newtheorem{mdlem}[thm]{Lemma}\crefname{mdlem}{lemma}{lemmas}
\newtheorem*{lem*}{Lemma}

\newtheorem{prop}[thm]{Proposition}\crefname{prop}{proposition}{propositions}
\newtheorem{mdprop}[thm]{Proposition}\crefname{mdprop}{proposition}{propositions}

\newtheorem{lemSublagrangian}[thm]{Sub-Lagrangian Reduction}\crefname{lemSublagrangian}{Sub-Lagrangian Reduction}{}
\newtheorem{lemFiltration}[thm]{Filtration Lemma}\crefname{lemFiltration}{the Filtration Lemma}{}
\newtheorem{lemGeneration}[thm]{Generation Lemma}\crefname{lemGeneration}{the Generation Lemma}{}
\newtheorem{lemLineGeneration}[thm]{Line Generation Lemma}\crefname{lemLineGeneration}{the Line Generation Lemma}{}

\theoremstyle{definition}
\newtheorem{defn}[thm]{Definition}\crefname{defn}{definition}{definitions}
\newtheorem*{defn*}{Definition}
\theoremstyle{remark}
\crefname{example}{example}{examples}
\newtheorem*{example*}{Example}
\newtheorem{rem}[thm]{Remark}\crefname{rem}{remark}{remarks}
\newtheorem*{rem*}{Remark}
\newtheorem*{warning*}{Warning}
\newtheorem*{hilfslem*}{Lemma}
\newcommand*{\xqed}[1]{\leavevmode\unskip\penalty9999 \hbox{}\nobreak\hfill\quad\hbox{#1}}
\newcommand*{\hilfsqed}{\xqed{$\triangle$}}
\newcommand*{\Z}{\mathbb{Z}}

\newcommand*{\C}{\mathbb{C}}

\newcommand*{\factor}[2]{\left.\raisebox{.1em}{\ensuremath{#1}}\middle/\raisebox{-.1em}{\ensuremath{#2}}\right.}
\providecommand*{\abs}[1]{\lvert#1\rvert}

\renewcommand*{\bar}{\overline}
\renewcommand*{\tilde}{\widetilde}

\DeclareMathOperator{\id}{id}

\DeclareMathOperator{\im}{im}
\DeclareMathOperator{\ind}{ind}

\DeclareMathOperator{\Hom}{Hom}
\DeclareMathOperator{\End}{End}
\DeclareMathOperator{\dimhom}{hom}
\DeclareMathOperator{\dimend}{end}

\newcommand*{\cc}[1]{#1^\bullet}

\newcommand*{\K}{\mathrm K}

\newcommand*{\GW}{\mathrm{GW}}

\newcommand*{\Rep}[1]{#1\text{-}\mathcal Rep}
\newcommand*{\KRep}[1]{\K(\Rep{#1})}
\newcommand*{\GWRep}[1]{\GW(\Rep{#1})}

\newcommand{\GL}[1]{\mathrm{GL}_{#1}}
\newcommand{\OO}[1]{\mathrm{O}_{#1}}
\newcommand{\SO}[1]{\mathrm{SO}_{#1}}
\newcommand{\Gm}{\mathbb G_m}
\newcommand{\twist}[2]{{^{#1}#2}}

\newcommand{\M}[1]{\begin{pmatrix}#1\end{pmatrix}}
\renewcommand*{\vec}[1]{\pmb{#1}}
\newcommand*{\cat}[1]{{\mathcal{#1}}}

\newcommand*{\dual}{\vee}

\newcommand*{\vecspan}[1]{\left\langle#1\right\rangle}

\newcommand{\ctext}[2]{\text{\parbox{#1}{\centering #2}}}
\newcommand{\ltext}[2]{\text{\parbox{#1}{\raggedright#2}}}

\newcommand{\lttext}[2]{\text{\parbox[t]{#1}{\raggedright#2}}}

\newcommand*{\IPS}[1]{\Lambda(#1)}

\newcommand{\mm}[1]{\left(\begin{smallmatrix}#1\end{smallmatrix}\right)}
\newcommand{\diag}[1]{\mathrm{diag}\left(#1\right)}

\newcolumntype{M}{>{$}c<{$}}

\entrymodifiers={+!!<0pt,\fontdimen22\textfont2>}
\SelectTips{cm}{10}
\newcommand*{\ie}{\mbox{i.\thinspace{}e.\ }}
\newcommand*{\eg}{\mbox{e.\thinspace{}g.\ }}
\newcommand*{\cf}{\mbox{c.\thinspace{}f.\ }}
\newcommand*{\define}[1]{\textbf{#1}}
\newcommand*{\GrothendieckWitt}{Grothen\-dieck-Witt }
\newcommand*{\safelambda}{\texorpdfstring{$\lambda$}{lambda}}

\begin{document}

\maketitle
\begin{abstract}
\noindent
   The representation ring of an affine algebraic group scheme can be endowed with the structure of a (special) \(\lambda\)-ring.
   We show that the same is true for the ring of symmetric representations, \ie for the \GrothendieckWitt ring of the representation category, for any affine algebraic group scheme over a field of characteristic not two.
\end{abstract}
\tableofcontents

\section*{Introduction}
It is well-known that both the complex and the real representation ring of any compact Lie group are \(\lambda\)-rings\footnote{%
By a \(\lambda\)-ring, we mean a ``special \(\lambda\)-ring''---see the paragraph on terminology below.} \cite{AtiyahTall}.
Similarly, for any affine algebraic group scheme \(G\) over a field, with representation category \(\Rep{G}\), the exterior power operations endow the representation ring \(\KRep{G}\) with the structure of a \(\lambda\)-ring.
Indeed, this is a direct consequence of Serre's beautiful 1968 paper ``Groupes de Grothendieck des schémas en groupes réductifs déployés'' \cite{Serre},  in which Serre shows that the representation ring of a split reductive group over an arbitrary field can be computed in the same way as---and is in fact isomorphic to---the representation ring of the corresponding group over \(\C\).  As Serre mentions in his introduction, establishing the \(\lambda\)-ring structure was in fact one of his motivations for writing the article.

The purpose of the present article is to complete the picture by establishing the \(\lambda\)-ring structure on the ``symmetric representation ring''\footnote{%
See footnote \ref{footnote:terminology} on page~\pageref{footnote:terminology} for a justification of this neologism.}
 \(\GWRep{G}\), generated by isotropy classes of representations equipped with equivariant non-degenerate symmetric forms.  This ring \(\GWRep{G}\) is to the usual representation ring \(\KRep{G}\) what the real representation ring is to the complex representation ring in topology.  See \Cref{sec:symRepRings} for precise definitions.  We will show:
\begin{thm*}
  For any affine algebraic group scheme \(G\) over a field of characteristic not two, the exterior power operations induce a \(\lambda\)-ring structure on the symmetric representation ring \(\GWRep{G}\).
\end{thm*}

To the best of our knowledge, this fundamental structure on \(\GWRep{G}\) has not been exposed before, except in the case when \(G\) is the trivial group:  the \(\lambda\)-ring structure on the Grothendieck-Witt ring of a field has been studied by McGarraghy \cite{McGarraghy:exterior}.

\paragraph{\(\lambda\)-Terminology.}
There are at least two problems with the term ``\(\lambda\)-ring''.  Firstly, the term is ambiguous: while Grothendieck originally distinguished between
\[
\begin{aligned}
&  &&\text{(1) ``\(\lambda\)-rings'' and}  && \text{(2) ``special \(\lambda\)-rings''} && \text{\cite{SGA6}*{Exposé~0~App},}
\shortintertext{Berthelot instead refers to these objects as}
&  &&\text{(1)  pre-\(\lambda\)-rings and}  && \text{(2) \(\lambda\)-rings} && \text{\cite{SGA6}*{Exposé~V}.}
\end{aligned}
\]
In this article, we follow Berthelot.  This seems to be the current trend, and it has the merit that the shorter term is reserved for the more natural object.  In any case, the bulk of this article is devoted to proving that we have a structure of type~(2), not just of type~(1).

Secondly, the term ``\(\lambda\)-ring'' is misleading in that it puts undue emphasis on \(\lambda\)-operations\slash exterior powers.  For example, from a purely algebraic perspective, the symmetric powers have just as good a claim to the title as the exterior powers.   We refer to \cite{Borger:Positivity} for a beautiful coordinate-free definition of \(\lambda\)-rings as ``rings equipped with all possible symmetric operations'' in a precise sense. 
Suffice it to remark here that the existence of a \(\lambda\)-structure on a ring includes the existence of many other natural operations such as symmetric powers and Adams operations.

That said, we will nevertheless work with the traditional definition in terms of exterior powers below. One technical reason for this is that exterior powers behave well under dualization:  the dual of the exterior power of a representation is the exterior power of the dual representation, in any characteristic.  The same is not true of symmetric operations.  Thus, in this case the existence of well-defined symmetric powers on \(\GWRep{G}\) follows only \emph{a posteriori} from the existence of a \(\lambda\)-structure.

\paragraph{Outline. }
The article begins with a certain amount of overhead.  We recall some definitions and  facts concerning symmetric representations, including a discussion of the additive structure of \(\GWRep{G}\) following Calmès and Hornbostel's preprint \cite{CalmesHornbostel:reductive}.

Our proof in \Cref{sec:lambdaRings} that the exterior powers induce well-defined maps on \(\GWRep{G}\) follows a similar pattern as the usual argument for \(\KRep{G}\), using in addition only the well-known technique of ``sub-Lagrangian reduction''.

When the ground field is algebraically closed, the fact that the resulting pre-\(\lambda\)-structure on \(\GWRep{G}\) is a \(\lambda\)-structure can easily be deduced in the same way as in topology:  in this case,  the forgetful map \(\GWRep{G}\to\KRep{G}\) exhibits \(\GWRep{G}\) as a sub-\(\lambda\)-ring of the \(\lambda\)-ring \(\KRep{G}\).  However, over general fields this argument breaks down.
\Cref{sec:structure-is-special} is devoted to mending it: we reduce to the ``universal case'', \ie the case when \(G\) is a product of split orthogonal groups, show that the symmetric representation ring of such \(G\) embeds into the symmetric representation ring of an extension of a maximal split torus, and verify that the latter is a \(\lambda\)-ring by a direct calculation.

The implications of the universal case are in fact not restricted to representation rings.  The main application we have in mind is to the Grothendieck-Witt ring of vector bundles on a scheme, in the same way that Serre's result is applied to the \(\K\)-ring of vector bundles in \cite{SGA6}*{Exposé~VI, Theorem~3.3}.  Details are to appear in forthcoming work.

\paragraph{Notation and conventions. }
Throughout, \(F\) denotes a fixed field of characteristic not two.  Our notation for group schemes, characters etc.\ tends to follow \cite{Jantzen}.  All representations are assumed to be finite-dimensional.

\section{Symmetric representations}\label{sec:symreps}
An affine algebraic group scheme is a functor \(G\) from the category of \(F\)-algebras to the category of groups representable by a finitely-generated \(F\)-algebra:
\[
\begin{aligned}
G\colon \cat Alg_F &\to \cat Groups \\
A &\mapsto G(A)
\end{aligned}
\]
We assume that the reader is familiar with the basic notions surrounding such group schemes and their representations as can be found in \cite{Waterhouse} and \cite{Jantzen} or \cite{Serre}.  In particular, while the basic notions involved in the statement of our main theorem are recalled below, we undertake no attempt to explain the structure and representation theory of reductive groups used in the proof.

The terms \define{representation} of \(G\) and \(G\)-\define{module} are used interchangeably to denote a \emph{finite-dimensional} \(F\)-vector space \(M\) together with a natural \(A\)-linear action of \(G(A)\) on \(M\otimes A\) for every \(F\)-algebra \(A\).
Equivalently, such a representation may be viewed as a group homomorphism
\(G\to \GL{}(M)\).
 Given two \(G\)-modules \(M\) and \(N\), the set of \(G\)-equivariant morphisms from \(M\) to \(N\) is denoted \(\Hom_G(M,N)\).

Many constructions available on vector spaces can be extended to \(G\)-modules.
In particular, \(G\)-modules form an \(F\)-linear abelian category \(\Rep{G}\).
Tensor products of \(G\)-modules, the dual \(M^\vee\) of a \(G\)-module \(M\) and its exterior powers \(\Lambda^i(M)\) are also again \(G\)-modules in a natural way.
There is, however, an important difference between the categories of \(G\)-modules and the category of vector spaces:  not every \(G\)-module is semi-simple, and a short exact sequence of \(G\)-modules does not necessarily split.

The duality functor \(M\mapsto M^\vee\) and the double-dual identification \(M\cong M^{\vee\vee}\) give \(\Rep{G}\) the structure of a category with duality, which immediately gives rise to the notion of a symmetric \(G\)-module in the sense of \cite{QSS}.  We hope there is no harm in providing a direct definition, even if we occasionally fall back into the abstract setting later on. We first discuss all relevant notions on the level of vector spaces.

\bigskip
A \define{symmetric vector space} is a vector space \(M\) together with a linear isomorphism \(\mu\colon M\to M^\vee\) which is symmetric in the sense that \(\mu\) and \(\mu^\vee\) agree up to the usual double-dual identification \(\omega\colon M\cong M^{\vee\vee}\).  The \define{orthogonal sum} \mbox{\((M,\mu)\perp(N,\nu)\)} of two symmetric vector spaces is defined as the direct sum \(M\oplus N\) equipped with the symmetry \(\mu\oplus\nu\).  Tensor products and exterior powers of symmetric vector spaces can be defined similarly, using the canonical isomorphisms \(M^\vee\otimes N^\vee\cong (M\otimes N)^\vee\) and \(\Lambda^i(M^\vee)\cong (\Lambda^i M)^\vee\).\footnote{%
In characteristic zero, one can likewise form symmetric powers \(S^i(M,\mu)\) of symmetric vector spaces.  However, we do not have a canonical isomorphism \(S^i(M^\vee)\cong (S^i M)^\vee\) in positive characteristic (\cf \cite{McGarraghy:symmetric} or \cite{Eisenbud}*{App.~A.2}).}

A morphism from  \((M,\mu)\) to \((N,\nu)\) is a morphism \(\iota\colon M\to N\) compatible with \(\mu\) and \(\nu\) in the sense that \(\iota^\vee\nu\iota = \mu\). An isomorphism with this property is an \define{isometry}. The isometries from \((M,\mu)\) to itself form a reductive subgroup \(\OO{}(M,\mu)\) of \(\GL{}(M)\).
If we equip \(F^{2n}\) and \(F^{2n+1}\) with the standard symmetric forms given by
\begin{equation}\label{eq:standard-symmetric-forms}
\mm{
  0&1& & & & \\
  1&0& & & & \\
  & &\dots&&&\\
  & & &\dots&&\\
  & & & &0&1\\
  & & & &1&0}
\quad\text{ and }\quad
\mm{
  0&1& & & & & \\
  1&0& & & & & \\
  & &\dots&&&&\\
  & & &\dots&&&\\
  & & & &0&1& \\
  & & & &1&0& \\
  & & & & & &1\\
}
\end{equation}
with respect to the canonical bases, we obtain the usual split orthogonal groups \(\OO{2n}\) and \(\OO{2n+1}\).

We also have a canonical symmetry on any vector space of the form \(M\oplus M^\vee\), given by interchanging the factors.
We write \(H(M) := (M\oplus M^\vee, \mm{0 & 1\\ 1 & 0})\) for this symmetric vector space; it is the \define{hyperbolic space} associated with \(M\).
The associated orthogonal group \(\OO{}(H(M))\) is isomorphic to \(\OO{2\dim M}\).

A \define{sub-Lagrangian} of a symmetric vector space \((M,\mu)\) is a subspace \(i\colon N\hookrightarrow M\) on which \(\mu\) vanishes, \ie for which \(i^\vee\mu i = 0\). Equivalently, if for an arbitrary subspace \(N\subset M\) we define
\[
N^\perp := \{m \in M \mid \mu(m)(n) = 0 \text { for all } n\in N\},
\]
then \(N\) is a sub-Lagrangian if and only if \(N\subset N^\perp\).
If in fact \(N^\perp = N\), we say that \(M\) is \define{metabolic} with \define{Lagrangian} \(N\).
For example, \(H(M)\) is metabolic with Lagrangian \(M\).

\bigskip
A \define{symmetric \(G\)-module} is defined completely analogously, as a pair \((M,\mu)\) consisting of a \(G\)-module \(M\) and an isomorphism of \(G\)-modules \(\mu\colon M\to M^\vee\) which is symmetric in the sense that \(\mu\) and \(\mu^\vee\) agree up to the double-dual identification of \(G\)-modules \(\omega\colon M\cong M^{\vee\vee}\).
Equivalently, we may view such a symmetric module
\begin{compactitem}
\item as a symmetric vector space \((M,\mu)\) together with a \(G\)-module structure on \(M\) such that \(\mu\) is \(G\)-equivariant, or
\item as a morphism \(G\to \OO{}(M,\mu)\), where \((M,\mu)\) is some symmetric vector space.
\end{compactitem}
All of the notions introduced for symmetric vector spaces carry over to this situation.

We end this section with a well-known lemma that makes use of the assumption that our field \(F\) has characteristic different from two.
\begin{lem}\label{hyperbolic-lemma}
For any symmetric \(G\)-module \((M,\mu)\), the orthogonal sum \((M,\mu)\oplus (M,-\mu)\) is isometric to the hyperbolic \(G\)-module \(H(M)\).
\end{lem}
\begin{proof}
\hfill\\
\(
  (M\oplus M, \mm{\mu & 0 \\ 0 & -\mu}) \cong_{\mm{1&1\\ \nicefrac{1}{2} & -\nicefrac{1}{2}}}
  (M\oplus M, \mm{0 & \mu\\ \mu & 0}) \cong_{\mm{1& 0 \\ 0 & \mu}}
  (M\oplus M^\vee, \mm{0 & 1 \\ 1 & 0})
\)
\end{proof}

\subsection{The symmetric representation ring}\label{sec:symRepRings}
The (finite-dimensional) representations of an affine algebraic group scheme over \(F\) form an abelian category with duality \((\Rep{G},\vee,\omega)\). Its \(\K\)-group and its \GrothendieckWitt group are defined as follows:

\begin{defn}
\(\KRep{G}\) is the free abelian group on isomorphism classes of \(G\)-modules
modulo the relation \(M = M' + M''\) for any short exact sequence of
\(G\)-modules \mbox{\(0\to M'\to M \to M''\to 0\).}

\(\GWRep{G}\) is the free abelian group on isometry classes of symmetric \(G\)-modules modulo the relation \(((M,\mu) \perp (N,\nu)) = (M\mu) + (N,\nu)\) for arbitrary symmetric \((M,\mu)\) and \((N,\nu)\) and the relation \((M,\mu) = H(L)\) for any metabolic \(G\)-module \(M\) with Lagrangian~\(L\).
\end{defn}
We use the established notation \(\K(F)\) and \(\GW(F)\) for the \(\K\)- and \GrothendieckWitt groups of the category of finite-dimensional vector spaces.  So \(\GW(F) = \GWRep{1}\), where \(1\) denotes the trivial constant group scheme.

The tensor product yields well-defined ring structures on both \(\KRep{G}\) and \(\GWRep{G}\).
The ring \(\KRep{G}\) is usually referred to as the \define{representation ring} of \(G\), and we refer to \(\GWRep{G}\) as the \define{symmetric representation ring}\footnote{\label{footnote:terminology}%
Of course, it is not the ring itself but rather its elements that are supposed to be symmetric.  However, our terminology is completely analogous to the established usage of the terms ``complex representation ring'' and ``real representation ring'' in the context of compact Lie groups. More precise alternatives would be ``ring of symmetric representations'' or ``symmetric representations' ring''.}
of \(G\).
They can be related via the forgetful and hyperbolic maps:
\[
\begin{aligned}
 \GWRep{G} &\xrightarrow{F}\KRep{G} \\
 \GWRep{G} &\xleftarrow[H]{} \KRep{G}
\end{aligned}
\]
The forgetful map simply sends the class of \((M,\mu)\) to the class of \(M\), while the hyperbolic map sends \(M\) to \(H(M)\). Note that \(F\) is a ring homomorphism, while \(H\) is only a morphism of groups.\footnote{In fact, \(H\) is a morphism of \(\GWRep{G}\)-modules, but we do not need this.}
We will need the following fact.

\begin{lemSublagrangian}\label{sublagrangian-reduction}
  For any sub-Lagrangian \(N\) of a symmetric \(G\)-module \((M,\mu)\), the symmetry \(\mu\) induces a symmetry \(\bar\mu\) on \(N^\perp/N\). Moreover, in \(\GWRep{G}\) we have
  \[
  (M,\mu) = (N^\perp/N, \bar\mu) + H(N).
  \]
\end{lemSublagrangian}
\begin{proof}
This is essentially Lemma~5.3 in \cite{QSS}, simplified by our assumption that \(2\) is invertible:
  \begin{align*}
    (M,\mu)
    &= -(N^\perp/N,-\bar\mu) + H(N^\perp) &&\text{ by \cite{QSS}*{Lemma~5.3}} \\
    &= (N^\perp/N,\bar\mu) - H(N^\perp/N) + H(N^\perp) &&\text{ by \Cref{hyperbolic-lemma}}\\
    &= (N^\perp/N,\bar\mu) + H(N) && \qedhere
  \end{align*}
\end{proof}

\subsection{The additive structure}
We recall some material from \cite{QSS} and \cite{CalmesHornbostel:reductive} concerning the additive structure of \(\KRep{G}\) and \(\GWRep{G}\).
The group \(\KRep{G}\) is the free abelian group on the isomorphism classes of simple \(G\)-modules.  Given a complete set \(\Sigma\) of representatives of the isomorphism classes of simple \(G\)-modules, we can thus write
\[
\KRep{G} \cong  \Z \vecspan{S}_{S\in\Sigma}.
\]
The structure of \(\GWRep{G}\) is slightly more interesting.
For simplicity, we  concentrate on the case when the endomorphism ring \(\End_G(S)\) of every simple \(G\)-module \(S\) is equal to the ground field \(F\).
This assumption is satisfied by all examples that we later study in more detail.  In particular, it is satisfied by all \(F\)-split reductive groups \cite{Jantzen}*{Cor.~II.2.7 and Prop.~II.2.8}.
It ensures that every simple \(G\)-module is either symmetric, anti-symmetric or not self-dual at all,
and that any two given (anti-)symmetries on a simple \(G\)-module differ at most by a scalar.

Let \(\Sigma_+\in \Sigma\) and \(\Sigma_-\in\Sigma\) be the subsets of symmetric and anti-symmetric objects,
and let  \(\Sigma_0\in\Sigma\) be a subset containing one object for each pair of non-self-dual objects \((S,S^\vee)\).
On each \(S\in\Sigma_+\), we fix a symmetry \(\sigma_s\).
\begin{thm}\label{GW-of-Rep:simplified}
Let \(G\) be an affine algebraic group scheme over \(F\) such that every simple \(G\)-module has endomorphism ring \(F\).
Then we have an isomorphism of  \(\GW(F)\)-modules
\[
\GWRep{G} \cong \GW(F)\vecspan{(S,\sigma_s)}_{S\in\Sigma_+} \oplus \Z\vecspan{H(S)}_{S\in\Sigma_-} \oplus \Z\vecspan{H(S)}_{S\in\Sigma_0}.
\]
\end{thm}
The theorem may require a few explanations.
The \(\GW(F)\)-module structure on \(\GWRep{G}\) is induced by the tensor product on \(\Rep{G}\) and the identification of the subcategory of trivial \(G\)-modules with the category of finite-dimensional vector spaces.  On the right-hand side, we can consider each copy of \(\GW(F)\) as a module over itself, with a canonical generator given by the trivial symmetry on \(F\).  The free abelian group \(\Z\) can be viewed as a \(\GW(F)\)-module via the rank homomorphism \(\GW(F)\to\Z\). As such, it is of course generated by \(1\in\Z\).
We can thus define a morphism of \(\GW(F)\)-modules
\[
\GWRep{G} \xleftarrow{\alpha} \bigoplus_{S\in\Sigma_+}\GW(F) \oplus \bigoplus_{S\in\Sigma_-}\Z \oplus \bigoplus_{S\in\Sigma_0} \Z
\]
that sends the canonical generator of the copy of \(\GW(F)\) corresponding to \(S\in\Sigma_+\) to \((S,\sigma_S)\) and the generator of the copy of \(\Z\) corresponding to \(S\in\Sigma_-\cup\Sigma_0\) to \(H(S)\).  The theorem says that this morphism is an isomorphism.

The inverse to \(\alpha\) can be described as follows.
A semi-simple \(G\)-module \(M\) can be decomposed into its \(S\)-isotypical summands \(M_S\).
A symmetry \(\mu\) on \(M\) necessarily decomposes into an orthogonal sum of its restrictions to \(M_S\) for each \(S\in\Sigma_+\) and its restrictions to  \(M_S\oplus M_{S^\vee}\) for each \(S\in\Sigma_-\cup\Sigma_0\). In fact, we can always find an isometry
\begin{equation}\label{eq:semisimple-decomp}
(M,\mu) \cong \bigoplus_{S\in\Sigma_+} \phi_S \cdot (S,\sigma_S) \oplus \bigoplus_{S\in\Sigma_-} n_S\cdot H(S) \oplus \bigoplus_{S\in\Sigma_0} m_S\cdot H(S)
\end{equation}
for certain symmetric forms \(\phi_S\) over \(F\) and non-negative integers \(n_S\) and \(m_S\).

In general, any symmetric \(G\)-module \((M,\mu)\) contains an isotropic \(G\)-submodule \mbox{\(N\subset M\)} such that \(N^\perp/N\) is semi-simple \cite{QSS}*{Theorem~6.10}. By \Cref{sublagrangian-reduction}, we then have
\[
(M,\mu) = (N^\perp/N, \bar\mu) + H(N) \quad\quad \text{ in }\GWRep{G}.
\]
The first summand can be decomposed as in \eqref{eq:semisimple-decomp}, and a decomposition of the second summand can be obtained from the decomposition of \(N\) into its simple factors in \(\KRep{G}\).
Thus, even for general \((M,\mu)\), in \(\GWRep{G}\) we have a decomposition of the form
\begin{equation}\label{eq:semisimple-decomp-GW}
(M,\mu) = \sum_{S\in\Sigma_+} \phi_S \cdot (S,\sigma_S) + \sum_{S\in\Sigma_-} n_S\cdot H(S) + \sum_{S\in\Sigma_0} m_S\cdot H(S)
\end{equation}
for certain symmetric forms \(\phi_S\) over \(F\) and non-negative integers \(n_S\) and \(m_S\).
In particular, \((M,\mu)\) decomposes into a sum, not a difference.

\begin{rem}[\cf \cite{CalmesHornbostel:reductive}*{Remark~1.15}]\label{symmetric-JH-decomp}
We can determine which summands in \eqref{eq:semisimple-decomp-GW} have non-zero coefficients from the decomposition of \(M\) in \(\KRep{G}\).
Indeed, the forgetful map \(\GW\to\K\) is compatible with the decompositions of \(\GWRep{G}\) and \(\KRep{G}\).  On the summand corresponding to \(S\in\Sigma_0\), it can be identified with the diagonal embedding \(\Z\hookrightarrow \Z\oplus\Z\), on the summand corresponding to \(S\in\Sigma_-\)  with multiplication by two \(\Z\hookrightarrow \Z\), and on the summand corresponding to \(S\in\Sigma_+\) with the rank homomorphism \(\GW(F)\to\Z\).  This last map is of course not generally injective,
but it does have the property that no non-zero symmetric form is sent to zero.
\end{rem}
We will use this observation to analyse the restriction \(\GWRep{\SO{m}}\to\GWRep{\OO{m}}\) in the proof of \Cref{thm:O-special}.

\section{The pre-\safelambda-structure}\label{sec:lambdaRings}
In this section we show that the exterior power operations define a pre-\(\lambda\)-structure on the symmetric representation rings \(\GWRep{G}\). We quickly recall the relevant definition from \cite{SGA6}*{Exposé~V, Définition~2.1}.

Given any commutative unital ring \(R\), we write \(\IPS{R} := (1 + tR[[t]])^\times\) for the multiplicative group of invertible power series over \(R\) with constant coefficient \(1\).
A \define{pre-\(\lambda\)-structure} on \(R\) is a collection of maps
\(\lambda^i\colon R\to R\), one for each \(i\in\mathbb{N}_0\),
such that \(\lambda^0\) is the constant map with value \(1\), \(\lambda^1\) is the identity, and the induced map
\[
\begin{aligned}
\lambda_t \colon R &\to \IPS{R}\\
r &\mapsto \textstyle\sum_{i\geq 0} \lambda^i(r)t^i
\end{aligned}
\]
is a group homomorphism.
A \define{pre-\(\lambda\)-ring} is a pair \((R,\lambda^\bullet)\) consisting of a ring \(R\) and a fixed such structure.
A morphism of pre-\(\lambda\)-rings \((R,\lambda^\bullet)\to(R',\lambda'^\bullet)\) is a ring homomorphisms that commutes with the maps \(\lambda^i\).
Following the terminology of Berthelot in loc.\ cit., we sometimes refer to such a morphism as a \define{\(\lambda\)-homomorphism} regardless of whether source and target are pre-\(\lambda\)-rings or in fact \(\lambda\)-rings (see \Cref{sec:lambda-rings}).

\begin{prop}\label{prop:lambda-on-GW}
  Let \(G\) be an affine algebraic group scheme over a field of characteristic not two.
  Then the exterior power operations  \( \lambda^k\colon (M, \mu) \mapsto (\Lambda^k M, \Lambda^k \mu) \) induce well-defined maps on \(\GWRep{G}\) which provide \(\GWRep{G}\) with the structure of a pre-\(\lambda\)-ring.
\end{prop}

We divide the proof into several steps, of which only the last differs somewhat from the construction of the \(\lambda\)-operations on \(\KRep{G}\).

\paragraph{Step 1.} We check that \(\lambda^i(M, \mu) := (\Lambda^i M, \Lambda^i\mu)\) is well-defined on the set of isometry classes of \(G\)-modules, so that we have an induced map
\[
\lambda_t\colon \left\{\ctext{2cm}{isometry classes of \(G\)-modules}\right\} \to \IPS{\GWRep{G}}.
\]
\paragraph{Step 2.} We check that \(\lambda_t\) is additive in the sense that
\[
\lambda_t((M, \mu)\perp(N,\nu)) = \lambda_t(M, \mu)\lambda_t(N,\nu).
\]
Then we extend \(\lambda_t\) linearly to obtain a group homomorphism
\[
\lambda_t\colon \bigoplus\Z (M, \mu) \to \IPS{\GWRep{G}},
\]
where the sum on the left is over all isometry classes of \(G\)-modules. By the additivity property, this extension factors through the quotient of \(\bigoplus\Z (M, \mu)\) by the ideal generated by the relations
\[
((M, \mu)\perp(N,\nu)) = (M, \mu) + (N,\nu).
\]
\paragraph{Step 3.} Finally, in order to obtain a factorization
\[
\lambda_t\colon \GWRep{G}\to \IPS{\GWRep{G}},
\]
we check that \(\lambda_t\) respects the relation \( (M,\mu) = H(L)\) for every metabolic \(G\)-bundle \( (M,\mu) \) with Lagrangian \(  L\).
For this step, we need the following refinement of the usual lemma used in the context of \(\K\)-theory (see for example \cite{SGA6}*{Exposé~V, Lemme~2.2.1}).
\begin{lemFiltration}\label{lemFiltration}
  Let \(0\to L\to M\to N\to 0\) be an extension of \(G\)-modules. Then we can find a filtration of \(\Lambda^n M\) by \(G\)-submodules
  \(
  \Lambda^n M = M^0 \supset M^1 \supset M^2 \supset \cdots
  \)
  together with isomorphisms of \(G\)-modules
  \begin{equation}\label{eq:lemFiltration}
  \factor{M^i}{M^{i+1}} \cong \Lambda^i L \otimes \Lambda^{n-i} N.
  \end{equation}
  More precisely, there is a unique choice of such filtrations and isomorphisms subject to the following conditions:
  \begin{compactenum}[(1)]
    \item\label{i:filt-1}
      The filtration is natural with respect to vector space isomorphisms of extensions. That is, given two extensions \(M\) and \(\tilde M\) of \(G\)-modules of \(L\) by \(N\), any vector space isomorphism \(\phi\colon M\to \tilde M\) for which
      \[\xymatrix{
        0 \ar[r] & L \ar[r]\ar@{=}[d] & M          \ar[r]\ar[d]^{\phi}_{\cong} & N \ar[r]\ar@{=}[d] & 0 \\
        0 \ar[r] & L \ar[r]           & {\tilde M} \ar[r]                    & N \ar[r]           & 0
      }\]
      commutes restricts to vector space isomorphisms \(M^i \to \tilde M^i\) compatible with \eqref{eq:lemFiltration} in the sense that
      \[\xymatrix{
        {\factor{M^i}{M^{i+1}}} \ar[rd]_{\cong}\ar[rr]_{\cong}^{\bar \phi} && \ar[ld]^{\cong} {\factor{\tilde M^i}{\tilde M^{i+1}}} \\
        & {\Lambda^i L\otimes \Lambda^{n-i} N}
      }\]
      commutes. In particular, the induced isomorphisms \(\bar\phi\) on the quotients are isomorphisms of \(G\)-modules.

    \item\label{i:filt-2}
      For the trivial extension, \((L\oplus N)^i  \subset \Lambda^n(L\oplus N)\) corresponds to the submodule
      \[
      \textstyle\bigoplus_{j\geq i} \Lambda^j L\otimes \Lambda^{n-j} N \quad \subset \quad \textstyle \bigoplus_j \Lambda^j L\otimes \Lambda^{n-j} N
      \]
      under the canonical isomorphism \(\Lambda^n(L\oplus N) \cong \bigoplus_j \Lambda^j L\otimes \Lambda^{n-j} N\), and the iso\-morphisms
      \[
      \factor{(L\oplus N)^i}{(L\oplus N)^{i+1}} \xrightarrow{\cong} \Lambda^i L\otimes \Lambda^{n-i} N
      \]
      correspond to the canonical projections.
  \end{compactenum}
\end{lemFiltration}
\begin{proof}[Proof of \cref{lemFiltration}]
Uniqueness is clear: if filtrations and isomorphisms satisfying the above conditions exist, they are determined on all split extensions by {\it (2)} and hence on arbitrary extensions by {\it (1)}.

Existence may be proved via the following direct construction.  Let \(0\to L\xrightarrow{\iota} M \xrightarrow{p} N \to 0\) be an arbitrary short exact sequence of \(G\)-modules.  Consider the \(G\)-morphism \(\Lambda^i L\otimes \Lambda^{n-i} M \to \Lambda^n M\) induced by \(\iota\).  Let \(M_i\) be its kernel and \(M^i\) its image, so that we have a short exact sequence of \(G\)-modules
\begin{equation}\label{eq:lemFiltration:proof}
\tag{$\ast$}
\xymatrix{
  0  \ar[r] &  {M_i} \ar[r] & {\Lambda^i L\otimes \Lambda^{n-i} M} \ar[r] & {M^i} \ar[r] & 0
}
\end{equation}
We claim that the images \(M^i\) define the desired filtration of \(M\).

Indeed, they define the desired filtration in the case \(M = L\oplus N\), and an isomorphism \(\phi \colon M \to \tilde M\) as in {\it (1)} induces (vector space) isomorphisms on each term of the corresponding exact sequences \eqref{eq:lemFiltration:proof}.  Moreover, the induced isomorphism on the central terms of these exact sequences is compatible with the projection to \(\Lambda^i L\otimes \Lambda^{n-i} N\).  The  situation is summarized by the following commutative diagram:
\[\xymatrix{
    0 \ar[r]
  & {M_i} \ar[r] \ar[d]^{\cong}
  & {\Lambda^i L\otimes \Lambda^{n-i} M} \ar[r] \ar[d]^{\cong} \ar@/_4ex/@{->>}[dd]|*+<2ex>{\hole}
  & {M^i} \ar[r] \ar[d]^{\cong} \ar@/^/@{..>}[ldd]
  & 0 \\
  0  \ar[r]
  & {\tilde M_i} \ar[r]
  & {\Lambda^i L\otimes \Lambda^{n-i} \tilde M} \ar[r]\ar@{->>}[d]
  & {\tilde M^i} \ar[r]\ar@{..>}[ld]
  & 0 \\
  &
  & {\Lambda^i L\otimes \Lambda^{n-i} N}
}
\]
We claim that the projection to \(\Lambda^i L\otimes \Lambda^{n-i} N\) factors through \(M^i\), as indicated by the dotted arrows.  This can easily be checked in the case of the trivial extension \(L\oplus N\).  In general, we may pick a vector space isomorphism \(\phi\colon M\to L\oplus N\) as in {\it (1)}.  Then the claim follows from the above diagram with \(L\oplus N\) in place of \(\tilde M\).  The same method shows that the induced morphisms \(M^i\to \Lambda^i L\otimes \Lambda^{n-i} N\) induce isomorphisms
\[
\factor{M^i}{M^{i+1}}\xrightarrow{\cong} \Lambda^i L\otimes \Lambda^{n-i} N.
\]
Note that while we use vector-space level arguments to verify that they are isomorphisms, they are, by construction, morphisms of \(G\)-modules.
\end{proof}

\begin{proof}[Proof of \Cref{prop:lambda-on-GW}, Step 1]
The exterior power operation
\[
\begin{aligned}
  \Lambda^i\colon &\Rep{G} \to \Rep{G}
\end{aligned}
\]
is a duality functor in the sense that we have a natural isomorphism \(\eta\) identifying
\(\Lambda^i(M^\vee)\) and \((\Lambda^i M)^\vee\) for each \(G\)-module \(M\).
Indeed,  we have natural isomorphisms of vector spaces
\[
\begin{aligned}
   \eta_M\colon \quad\Lambda^i(M^\vee)\quad &\xrightarrow{\cong} \quad \big(\Lambda^i M\big)^\vee\\
   \phi_1\wedge\dots\wedge\phi_i &\mapsto \big(m_1\wedge\dots\wedge m_i \mapsto \det(\phi_\alpha(m_\beta))\big)
\end{aligned}
\]
\citelist{\cite{Eisenbud}*{Prop.~A.2.7}\cite{Bourbaki:Algebre}*{Ch.~3, \S\,11.5, (30 bis)}}. These isomorphisms are equivariant with respect to the \(G\)-module structures induced on both sides by a \(G\)-module structure on \(M\).
We therefore obtain a well-defined operation on the set of isometry classes of symmetric \(G\)-modules by defining
\[
\lambda^i(M,\mu) := (\Lambda^i M, \eta_M\circ \Lambda^i(\mu)).
\]
Note however that the functor \(\Lambda^i\) is not additive or even exact, so it does not induce a homomorphism \(\GWRep{G}\to\GWRep{G}\).

\proof[Proof of \Cref{prop:lambda-on-GW}, Step 2]
In order to verify the claimed additivity property of \(\lambda_t\),  we need to check that, for any pair of symmetric \(G\)-modules \((M,\mu)\) and \((N,\nu)\), the natural isomorphism
\[
  \Lambda^n(M\oplus N) \xleftarrow{\cong} \bigoplus_{i}\Lambda^i M \otimes \Lambda^{n-i} N
\]
defines an \emph{isometry}
\[
  \lambda^n((M,\mu)\perp (N,\nu)) \xleftarrow{\cong} \bigoplus_{i}\lambda^i(M,\mu) \otimes \lambda^{n-i} (N,\nu).
\]
Denoting the \(i^{\text{th}}\) component of this natural isomorphism by \(\Phi_i\), the claim boils down to the commutativity of the following diagrams (one for each i), which can be checked by a direct computation.
\[
\xymatrix{
  {\Lambda^n(M^\vee\oplus N^\vee)} \ar[dd]_{\eta_{N\oplus M}} &
  \ar[l]_-{\Phi_i} {\Lambda^i (M^\vee)\otimes\Lambda^{n-i} (N^\vee)} \ar[d]^{\eta_M\otimes\eta_N}
  \\
  &
  {(\Lambda^i M)^\vee \otimes (\Lambda^{n-i} N)^\vee} \ar[d]^{\cong}
  \\
  {(\Lambda^n(M\oplus N))^\vee} \ar[r]_-{\Phi_i^\vee} &
  {(\Lambda^i M\otimes\Lambda^{n-i} N)^\vee}
}
\]

\proof[Proof of \Cref{prop:lambda-on-GW}, Step 3]
  Let \( (M,\mu) \) be metabolic with Lagrangian \(L\), so that we have a short exact sequence
  \begin{equation}\label{eq:lambda-on-GW:ses}
  0 \to L \xrightarrow{i} M \xrightarrow{i^\dual\mu} L^\dual \to 0.
  \end{equation}
  We need to show that \(\lambda^n (M, \mu) = \lambda^n H(L) \) in \(\GWRep{G}\).

  On the level of vector spaces, the exact sequence \eqref{eq:lambda-on-GW:ses} necessarily splits. In fact, we can find an \emph{isometry} of vector spaces \(\phi\colon (M,\mu) \to H(L)\) such that the diagram
  \[\xymatrix{
    0 \ar[r] & L \ar[r]\ar@{=}[d] & M          \ar[r]\ar[d]^{\phi}_{\cong} & {L^\vee} \ar[r]\ar@{=}[d] & 0 \\
    0 \ar[r] & L \ar[r]           & {H(L)}     \ar[r]                    & L^\vee \ar[r]           & 0
  }\]
  of \cref{lemFiltration}, {\it (1)} commutes.  For example, given any splitting \(s\) of \(i^\vee\mu\), let \(\tilde s\) be the alternative splitting \(\tilde s := s - \tfrac{1}{2}is^\vee\mu s\) and define \(\phi\) to be the inverse of \((i,\tilde s)\).
 We then have filtrations  \(M^\bullet\) and \(H(L)^\bullet\) of \(\Lambda^n M\) and \(\Lambda^n H(L)\) such that the isometry \(\Lambda^n\phi\) restricts to isomorphisms \(M^i\cong H(L)^i\) and induces isomorphisms of \(G\)-modules
  \[
  \factor{M^i}{M^{i+1}} \xrightarrow[\cong]{\bar\phi} \factor{H(L)^i}{H(L)^{i+1}}.
  \]

If \(n\) is odd, say \(n=2k-1\), then \(H(L)^k\) is a Lagrangian of \(\lambda^n H(L)\) and hence \(M^k\) is a Lagrangian of \(\lambda^n (M,\mu)\).
Therefore, in \(\GWRep{G}\) we have:
  \[
\begin{aligned}
    \lambda^n (M,\mu) &= H(M^k)  \\
    \lambda^n H(L) &= H(H(L)^k)
  \end{aligned}
\]
On the other hand, \(M^k = H(L)^k\) in  \(\KRep{G}\), since these two \(G\)-modules have filtrations with isomorphic quotients. So the right-hand sides of the above two equations agree, and the desired equality \(\lambda^n (M,\mu) = \lambda^n H(L)\) in \(\GWRep{G}\) follows.

If \(n\) is even, say \(n=2k\), then \(H(L)^{k+1}\) is a sub-Lagrangian of \(\lambda^n H(L)\), and \((H(L)^{k+1})^\perp = H(L)^k\).
Again, it follows from the fact that \(\phi\) is an isometry that likewise \(M^{k+1}\) is an admissible sub-Lagrangian of \(\lambda^n (M,\mu)\), and that \((M^{k+1})^\perp = M^k\).
Moreover, \(\phi\) induces an isometry of symmetric \(G\)-modules
\[
\left(\factor{M^k}{M^{k+1}},\bar\mu\right) \cong \left(\factor{H(L)^k}{H(L)^{k+1}},\bar{\mm{0&1\\1&0}}\right)
\]
The desired identity in \(\GWRep{G}\) follows:
\begin{align*}
  \lambda^{2k}(M,\mu)
  &= H(M^k) + (M^k/M^{k+1},\bar\mu)
  \quad\quad\text{(by \Cref{sublagrangian-reduction})} \\
  &= H(H(L)^k) + (H(L)^k/H(L)^{k+1},\bar{\mm{0&1\\1&0}})\\
  &= \lambda^{2k}H(L) \qedhere
\end{align*}
\end{proof}

Note that, by construction, \(\lambda^0=1\) (constant), \(\lambda^1=\id\), and \(\lambda_t\) is a ring homomorphism.  Thus \(\GWRep{G}\) is indeed a pre-\(\lambda\)-ring.
We observe a few additional structural properties.

\begin{defn}
  An \define{augmentation} of a pre-\(\lambda\)-ring \(R\) is a \(\lambda\)-homomorphism
  \[
  d\colon R\to \Z,
  \]
  where the pre-\(\lambda\)-structure on \(\Z\) is defined by \(\lambda^i(n) := \binom{n}{i}\).

  A \define{positive structure} on a pre-\(\lambda\)-ring \(R\) with augmentation \(d\) is a subset \(R_{>0}\subset R\) satisfying the axioms below.\footnote{Definitions in the literature vary.  The last of the axioms we require here, introduced by Grinberg in \cite{Grinberg}, appears to be missing in both \cite{FultonLang}*{\S\,I.1} and \cite{Weibel:K}*{Def.~II.4.2.1.}.  Without it, it is not clear that the set of line elements forms a subgroup of the group of units of \(R\). }
 Elements of \(R_{>0}\) are referred to as \define{positive elements}; a \define{line element} is a positive element \(l\) with \(d(l) = 1\).  The axioms are as follows:
\begin{align*}
  &\left.\begin{aligned}
  &\bullet\; \text{\parbox[t]{0.9\textwidth}{\(R_{\geq 0} := R_{>0} \cup \{0\}\) is closed under addition, under multiplication and under the \(\lambda\)-operations.}}\\
  &\bullet\; \text{Every element of \(R\) can be written as a difference of positive elements.}\\
  &\bullet\; \text{The element \(n\cdot 1_R\) is positive for every \(n\in \Z_{>0}\).}
  \end{aligned}
  \right.\\
  &\left.
    \begin{aligned}
      &\bullet\; d(x) > 0 \\
      &\bullet\; \lambda^{d(x)}(x) \text{ is a unit in } R \\
      &\bullet\; \lambda^i x = 0 \text{ for all } i>d(x)
    \end{aligned}
  \;\right\} \text{ for all positive elements \(x\)}\\
  &\left.\begin{aligned}
  &\bullet\; \text{\parbox[t]{0.9\textwidth}{The multiplicative inverse of a line element is a positive element (and hence again a line element).}}
  \end{aligned}
  \right.
  \end{align*}
\end{defn}
On \(\GWRep{G}\), we can define a positive structure by taking \(d(M,\mu) := \dim(M)\) and letting \(\GWRep{G}_{>0}\subset \GWRep{G}\) be the image of the set of isometry classes of \(G\)-modules in \(\GWRep{G}\). Then the line elements are the classes of symmetric characters of \(G\).

\begin{defn}\label{defn:line-special}
A pre-\(\lambda\)-ring \(R\) with a positive structure is \define{line-special} if
\[
   \lambda^k(l\cdot x) = l^k\lambda^k(x)
\]
for all line elements \(l\), all elements \(x\in R\) and all positive integers \(k\).
\end{defn}

\begin{lem}\label{lem:line-special}
The symmetric representation ring \(\GWRep{G}\) of an affine algebraic group scheme is line-special.
\end{lem}
\begin{proof}
It suffices to check this property on a set of additive generators of the \(\lambda\)-ring, for example on all positive elements.
Thus, it suffices to check that for any one-dimensional symmetric representation \((O,\omega)\) and any symmetric representation \((M,\mu)\), the canonical isomorphism \(O^{\otimes k}\otimes\Lambda^k M \cong \Lambda^k(O\otimes M)\) extends to an isometry \((O,\omega)^{\otimes k}\otimes\lambda^k(M,\mu)\cong \lambda^k((O,\omega)\otimes (M,\mu))\).
\end{proof}

\section{The pre-\safelambda-structure is a \safelambda-structure}\label{sec:structure-is-special}
Having established a pre-\(\lambda\)-structure on \(\GWRep{G}\), our aim is to show that it is in fact a \(\lambda\)-structure. We briefly recall the definition and some general facts before focusing on \(\GWRep{G}\) from \Cref{sec:reduction} onwards.

\subsection{\safelambda-rings}\label{sec:lambda-rings}
A pre-\(\lambda\)-ring \(R\) is a \define{\(\lambda\)-ring} if the group homomorphism
\[
\lambda_t\colon R \to \IPS{R}
\]
is in fact a \(\lambda\)-homomorphism, for a certain universal pre-\(\lambda\)-ring structure on \(\IPS{R}\) \cite{SGA6}*{Exposé~V, Définition~2.4.1\footnote{There are four different choices of multiplication on \(\IPS{R}\) that yield isomorphic ring structures, with respective multiplicative units of the form \((1\pm t)^{\pm 1}\).  We stick to loc.\ cit.\ and use the multiplication whose unit is \(1+t\).}}.
This property can be encoded by certain universal polynomials
\[
\begin{aligned}
  P_k &\in \Z[x_1,\dots,x_k,y_1,\dots,y_k] \\
  P_{k,j} &\in \Z[x_1,\dots,x_{kj}]
\end{aligned}
\]
as follows:  a pre-\(\lambda\)-ring \(R\) is a \(\lambda\)-ring if and only if \(\lambda_t(1) = 1 + t\) and\footnote{For a pre-\(\lambda\)-ring with a positive structure, \(\lambda_t(1) = 1+t\) is automatically satisfied.}
\begin{align}
  \label{eq:special-1}\tag{$\lambda 1$}
  \lambda^k(x\cdot y) &= P_k(\lambda^1 x,\dots,\lambda^k x, \lambda^1y,\dots,\lambda^ky )\\
  \label{eq:special-2}\tag{$\lambda 2$}
  \lambda^k(\lambda^j(x)) &= P_{kj}(\lambda^1 x,\dots,\lambda^{kj} x)
\end{align}
for all \(x,y\in R\) and all positive integers \(j,k\).
We refer to the equations \eqref{eq:special-1} and \eqref{eq:special-2} as the first and second \(\lambda\)-identity.  Precise definitions of the polynomials \(P_k\) and \(P_{k,j}\) are given in equations \eqref{eq:universal-polynomials-1} and \eqref{eq:universal-polynomials-2} below. Essentially, the \(\lambda\)-identities say that any element behaves like a sum of line elements.  A morphism of \(\lambda\)-rings is the same as a morphism of the underlying pre-\(\lambda\)-rings, \ie a ring homomorphism that commutes with the \(\lambda\)-operations.  We continue to refer to such morphisms as \define{\(\lambda\)-homomorphisms}.

Let us recall a few general criteria for verifying that a pre-\(\lambda\)-ring \(R\) with a positive structure is a \(\lambda\)-ring.

\begin{description}
\item[Embedding] If we can enlarge \(R\) to a \(\lambda\)-ring, \ie if we can find a \(\lambda\)-ring \(R'\) and a \(\lambda\)-monomorphism \(R\hookrightarrow R'\), then \(R\) itself is a \(\lambda\)-ring.

\item[Splitting] If all positive elements of \(R\) decompose into sums of line elements, then \(R\) is a \(\lambda\)-ring.

\item[Generation] If \(R\) is additively generated by elements satisfying the \(\lambda\)-identities, then \(R\) is a \(\lambda\)-ring.

More generally, if \(R\) is generated by line elements over some set of elements that satisfy the \(\lambda\)-identities, and if \(R\) is line-special, then \(R\) is a \(\lambda\)-ring.
Precise definitions are given below.

\item[Detection] If an element \(x\in R\) lies in the image of a \(\lambda\)-ring \(R'\) under a  \(\lambda\)-morphism \(R'\to R\), then the second \(\lambda\)-identity \eqref{eq:special-2} is satisfied for \(x\). Likewise, if two elements \(x,y\in R\) simultaneously lie in the image of a \(\lambda\)-ring \(R'\to R\), then both \(\lambda\)-identities \eqref{eq:special-1} and \eqref{eq:special-2} are satisfied for \(\{x, y\}\).

This criterion is particularly useful in combination with the generation criterion:  in order to show that a pre-\(\lambda\)-ring is a \(\lambda\)-ring, it suffices to check that each pair of elements from a set of additive generators is contained in the image of some \(\lambda\)-ring.
\end{description}

The embedding and detection criteria are easily verified directly from the definition of a \(\lambda\)-ring in terms of the \(\lambda\)-identities.  The splitting criterion follows from the generation criterion and the first part of \Cref{lambda-identities-for-low-rank-elements} below.  We discuss the generation criterion in some detail:

\begin{lemGeneration}\label{lem:special-via-generators}
  Let \(R\) be a pre-\(\lambda\)-ring with a positive structure, and let \(E\subset R\) be a subset that generates \(R\) as an abelian group (\eg \(E = R_{\geq 0}\)).
  Then \(R\) is a \(\lambda\)-ring if and only if the \(\lambda\)-identities \eqref{eq:special-1} and \eqref{eq:special-2} hold for all elements of \(E\).
\end{lemGeneration}
\begin{proof} 
In general, given any group homomorphism between rings \(l\colon R\to L\) and a subset \(E\subset R\) that generates \(R\) as a an abelian group, \(l\) is a morphism of rings if and only if it maps \(1\) to \(1\) and \(e_1e_2\) to \(l(e_1)l(e_2)\) for all elements \(e_1,e_2\in E\). Likewise, if \(R\) and \(L\) are pre-\(\lambda\)-rings, then \(l\) is a morphism of pre-\(\lambda\)-rings if and only if it maps \(e_1e_2\) to \(l(e_1)l(e_2)\) and \(\lambda^i(e)\) to \(\lambda^i(l(e))\) for all \(e_1,e_2,e\in E\) and all \(i\in \mathbb N\).
The lemma is proved by applying these observations to \(\lambda_t\colon R \to \IPS{R}\).
The assumption that \(R\) has a positive structure is needed only to verify that \(\lambda_t\) sends the multiplicative unit \(1\in R\) to the multiplicative unit \(1+t\in \IPS{R}\).
\end{proof}

Both the generation criterion and the splitting criterion are special cases of the following Line Generation Lemma.
\begin{defn}
  Let \(E\subset R\) be a subset of a pre-\(\lambda\)-ring with a positive structure.
  We say that \(R\) is
\define{generated by line elements over \(E\)} if every element of \(R\) can be written as a finite sum
\[
  \sum l_e\cdot e
\]
for certain elements \(e\in E\) and certain line elements \(l_e\) in \(R\).
\end{defn}

\begin{lemLineGeneration}\label{lem:special-via-line-generators}
Let \(R\) be a pre-\(\lambda\)-ring generated by line elements over some subset \(E\).
If \(R\) is line-special and if the \(\lambda\)-identities hold for all elements of \(E\),
then \(R\) is a \(\lambda\)-ring.
\end{lemLineGeneration}

The proofs of this lemma and the next are the only places where we will need the definitions of the polynomials \(P_k\) and \(P_{k,j}\).  Given a tuple \(\vec x = (x_1,\dots,x_n)\), let \(\lambda_i(\vec x)\) denote the \(i^{\text{th}}\) elementary symmetric polynomial in its entries.
The polynomials \(P_k\) and \(P_{k,j}\) are uniquely determined by the requirement that the following equations be satisfied in \(\Z[x_1,\dots,x_n]\), for all \(n\):
\begin{align}
  \label{eq:universal-polynomials-1}
  \sum_{k\geq 0} P_k(\lambda_1(\vec x),\dots,\lambda_k(\vec x), \lambda_1(\vec y), \dots, \lambda_k(\vec y)) T^k &= \prod_{\mathclap{1\leq i,j \leq n}}(1 + x_i y_j T) \\
  \label{eq:universal-polynomials-2}
  \sum_{k\geq 0} P_{k,j}(\lambda_1(\vec x),\dots,\lambda_{kj}(\vec x))T^k &= \prod_{\mathclap{1\leq i_1<\cdots<i_j\leq n}} (1+x_{i_1}\cdots x_{i_j}T) \end{align}

\begin{proof}[Proof of \cref{lem:special-via-line-generators}]
We claim that the following equations hold in \(\Z[\alpha,x_1,\dots,x_k,\beta,y_1,\dots,y_k]\) and in \(\Z[\alpha,x_1,\dots,x_k]\):
\begin{align}
  \label{eq:Pk-line-special}
  P_k(\alpha x_1,\dots, \alpha^k x_k, \beta y_1, \dots, \beta^k y_k) &= \alpha^k\beta^k P_k(x_1,\dots,x_k,y_1,\dots,y_k)\\
  \label{eq:Pkj-line-special}
  P_{k,j}(\alpha x_1,\alpha^2 x_2 \dots,\alpha^{kj} x_{kj}) &= \alpha^{kj} P_{k,j}(x_1,\dots,x_k)
\end{align}
Indeed, this follows easily from the fact that \(\lambda_i(\alpha\vec x)=\alpha^i\lambda_i(\vec x)\) by comparing the coefficients of \(T^k\) in the defining equations.
Let us now apply \cref{lem:special-via-generators} to the subset
\[
E' := \{ l e \mid e\in E, l \text{ a line element in } R\}.
\]
We check that all elements of \(E'\) satisfy the \(\lambda\)-identities:
for \(e\in E\) and \(l\in R\), we have
\[
\begin{aligned}
\lambda^k(\lambda^j(l e))
&= l^{kj}\lambda^k(\lambda^j e)         && \text{ since } R \text{ is line-special}\\
&= l^{kj}P_{k,j}(e, \lambda^2e, \dots, \lambda^{kj} e) && \text{ by the assumption on } E\\
&= P_{k,j}(l e, l^2\lambda^2e, \dots, l^{kj}\lambda^{kj} e) && \text{ by \cref{eq:Pkj-line-special}}\\
&= P_{k,j}(l e, \lambda^2(l e), \dots, \lambda^{kj}(l e)) && \text{ since } R \text{ is line-special}
\end{aligned}
\]
Similarly, for \(e_1,e_2\in E\) and any line elements \(l_1,l_2\in R\) we have
\begin{align*}
   \lambda^k(l_1 e_1\cdot l_2 e_2) &= P_k(\lambda^1(l_1 e_1), \dots, \lambda^k(l_1 e_1), \lambda^1(l_2 e_2), \dots, \lambda^k(l_2 e_2)).\qedhere
\end{align*}
\end{proof}

\begin{lem}\label{lambda-identities-for-low-rank-elements}
Let \(K\) be a pre-\(\lambda\)-ring with a positive structure.
  \begin{compactenum}[(i)]
  \item \label{lambda-identities-for-rank-one}
    The \(\lambda\)-identities \eqref{eq:special-1} and \eqref{eq:special-2} are satisfied by arbitrary line elements.
  \item \label{lambda-identities-for-rank-two}
    The \(\lambda\)-identities \eqref{eq:special-1} and \eqref{eq:special-2} are satisfied by a pair of positive elements \(x\) and \(y\) both of rank at most two if and only if the identities \eqref{eq:special-1} hold for \(k\in\{2,3,4\}\). Explicitly, for positive \(x\) and \(y\) of rank at most two said identities read as follows:
    \[
    \begin{aligned}
      \lambda^2(xy) &= x^2\cdot \lambda^2y + y^2 \cdot \lambda^2 x - 2 \lambda^2x\lambda^2y\\
      \lambda^3(xy) &= xy\cdot\lambda^2x\cdot\lambda^2 y\\
      \lambda^4(xy) &= (\lambda^2x)^2\cdot(\lambda^2y)^2
    \end{aligned}
    \]
  \item \label{lambda-identities-vs-line-special}
    \(K\) is line-special if and only if the identity \eqref{eq:special-1} is satisfied for any pair of elements \(x,y\in K\)  with \(x\) a line element.
  \end{compactenum}
\end{lem}

\begin{proof}
We sketch the proof of part~\eqref{lambda-identities-for-rank-two}. Consider first the identities \eqref{eq:special-2}. If we set all variables \(x_3, x_4,\dots\) to zero in the defining equations \eqref{eq:universal-polynomials-2} for \(P_{k,j}\), we obtain the identities
\[
\begin{aligned}
\textstyle\sum_k P_{k,1}(\lambda_1x, \lambda_2x, 0,\dots, 0)T^k
&= (1+x_1T)(1+x_2T) = 1 + \lambda_1x \cdot T + \lambda_2 x \cdot T^2\\
\textstyle\sum_k P_{k,2}(\lambda_1x, \lambda_2x, 0,\dots, 0)T^k
&= 1+x_1x_2\cdot T = 1 + \lambda_2x \cdot T\\
\textstyle\sum_k P_{k,j}(\lambda_1x, \lambda_2x, 0,\dots, 0)T^k &= 1 \quad\quad\text{ for all } j \geq 3
\end{aligned}
\]
Thus, for any element \(x\in K\) satisfying \(\lambda^kx = 0\) for \(k\geq 3\), the identities \eqref{eq:special-2} may be written as
\[
\begin{aligned}
\lambda^1(\lambda^1 x) &= \lambda^1 x, \quad
\lambda^2(\lambda^1 x) = \lambda^2 x,  \quad
\lambda^k(\lambda^1 x) = 0 \text{ for all } k \geq 3\\
\lambda^1(\lambda^2 x) &= \lambda^2 x, \quad
\lambda^k(\lambda^2 x) = 0 \text{ for all } k \geq 2\\
\lambda^k(\lambda^j x) &= 0 \text{ for all } j \geq 3 \text{ and all } k \geq 1
\end{aligned}
\]
If \(x\) is positive of rank at most two, then \(\lambda^2x\) is positive of rank at most one and all these relations are trivial.

Similarly, if we set all variables \(x_3, x_4,\dots\) and \(y_3, y_4, \dots\) to zero in the defining equation \eqref{eq:universal-polynomials-1} for the polynomials \(P_{k}\), we obtain the following identity:
\[
\begin{aligned}
  \textstyle\sum_k &P_k(\lambda_1x,\lambda_2x,0,\dots,0,\lambda_1y, \lambda_2y, 0,\dots, 0)\\
  &= (1 + x_1 y_1 T)(1 + x_1 y_2 T)(1 + x_2 y_1 T)(1 + x_2 y_2 T)\\
  &= 1 + \left(\lambda_1 x \cdot \lambda_1 y\right) T \\
  &\quad\quad + \left((\lambda_1 x)^2\cdot\lambda_2y + (\lambda_1 y)^2\cdot\lambda_2 x - 2\lambda_2x\cdot\lambda_2y\right) T^2\\
  &\quad\quad + \left(\lambda_1x \cdot \lambda_1 y \cdot \lambda_2 x \cdot \lambda_2y \right) T^3
       + \left((\lambda_2 x)^2\cdot (\lambda_2 y)^2\right) T^4
\end{aligned}
\]
The claims follow.
\end{proof}

\subsection{Reduction to the case of split orthogonal groups}\label{sec:reduction}
Our goal is to show that the pre-\(\lambda\)-structure on the symmetric representation ring of an affine algebraic group scheme defined above is in fact a \(\lambda\)-structure.
As a first step, we reduce to the case of the split orthogonal group \(\OO{m}\) and its products \(\OO{m_1}\times\OO{m_2}\).

For comparison and later use, we recall from \cite{SGA6}*{Exposé~0, App.\ RRR, \S\,2, 1) and 3)} the corresponding argument for the usual representation rings: the fact that these are \(\lambda\)-rings for any affine algebraic group scheme follows from the case of products of general linear groups \(\GL{m_1}\times\GL{m_2}\).
\begin{thm}[Serre]\label{G-K-special}
  The representation ring \(\KRep{G}\) of any affine algebraic group scheme \(G\) is a \(\lambda\)-ring.
\end{thm}
\begin{proof}[Proof, assuming the theorem for \(\GL{m_1}\times\GL{m_2}\)]
  Any finite-dimensional linear \(G\)-module can be obtained by pulling pack the standard representation of \(\GL{m}\) along some morphism \(G\to \GL{m}\). Its class in \(\KRep{G}\) is therefore contained in the image of the induced morphism of \(\lambda\)-rings
  \(\KRep{\GL{m}} \to \KRep{G}\).

  Similarly, given two \(G\)-modules corresponding to morphisms \(G\xrightarrow{\rho_1}\GL{m_1}\) and \(G\xrightarrow{\rho_2}\GL{m_2}\), we can consider the composition
  \[
  G\to G\times G \xrightarrow{(\rho_1,\rho_2)} \GL{m_1}\times\GL{m_2},
  \]
  where the first map is the diagonal.  Under this composition, the standard representation of \(\GL{m_1}\) pulls back to the first \(G\)-module, while the standard representation of \(\GL{m_2}\) pulls back to the second. Thus, the classes of these \(G\)-modules are both contained in the image of the induced morphism
  \[
  \KRep{\GL{m_1}\times \GL{m_2}} \to \KRep{G}.
  \]
  Therefore, by the detection criterion, it suffices to know that \(\KRep{\GL{m_1}\times\GL{m_2}}\) is a \(\lambda\)-ring.
\end{proof}
For the case of \(\GL{m_1}\times\GL{m_2}\) itself, see for example \cite{Serre} and the remarks below.

\begin{thm}\label{G-GW-special}
  The symmetric representation ring \(\GWRep{G}\) of any affine algebraic group scheme \(G\) over a field of characteristic not two is a \(\lambda\)-ring.
\end{thm}
\begin{proof}[Proof, assuming the theorem for \(\OO{m_1}\times\OO{m_2}\)]
  A symmetric representation \((\rho_E,\epsilon)\) of \(G\) corresponds to a morphism
  \(
  G \to \OO{}(E,\epsilon),
  \)
  where \((E,\epsilon)\) is some symmetric vector space. Of course, in general \((E,\epsilon)\) will not be split, but we can achieve this as follows:

  As \(2\) is invertible, the orthogonal sum \((E,\epsilon)\oplus (E,-\epsilon)\) is isometric to the hyperbolic space \(H(E)\) (\Cref{hyperbolic-lemma}).
  Let \(c_E\) denote the trivial \(G\)-module with underlying vector space \(E\).
  Then the symmetric \(G\)-module \((\rho_E,\epsilon) \oplus (c_E, -\epsilon)\) corresponds to a morphism
  \[
\begin{aligned}
   G &\to \OO{}\left((E,\epsilon)\oplus (E,-\epsilon)\right) \cong \OO{}\left(H(E)\right) \cong \OO{2\dim E}.
  \end{aligned}
\]
  Thus, the class of  \((\rho_E,\epsilon) \oplus (c_E, -\epsilon)\) is contained in the image of a morphism of \(\lambda\)-rings
  \begin{equation*}\label{eq:red-to-O}
  \GWRep{\OO{2\dim E}}\to \GWRep{G}.
  \end{equation*}
  The second summand, the trivial representation \((c_E, -\epsilon)\), can be obtained by pulling back the corresponding trivial representation from \(\OO{2\dim E}\). So the first summand, the class of \((\rho_E,\epsilon)\), is itself in the image of \eqref{eq:red-to-O}.

Likewise, given two symmetric representations \((\rho_E,\epsilon)\) and \((\rho_F,\phi)\) of \(G\), we can obtain
\( (\rho_E,\epsilon) \oplus (c_E,-\epsilon) \) and \((\rho_F,\phi) \oplus (c_F,-\phi) \) by restricting from \(\OO{2\dim E}\times \OO{2\dim F}\); both \((\rho_E,\epsilon)\) and \((\rho_F,\phi)\) are therefore contained in the image of a morphism
\[
  \GWRep{\OO{2\dim E}\times\OO{2\dim F}} \to \GWRep{G}.
\]
So under the assumption that \(\GWRep{\OO{2\dim E}\times\OO{2\dim F}}\) is a \(\lambda\)-ring, we can conclude as in the proof of \Cref{G-K-special}.
\end{proof}
It remains to show that the theorem is indeed true for products of split orthogonal groups.
This is the aim of the following sections, finally achieved in \Cref{thm:O-special}.

\subsection{Outline of the proof for split orthogonal groups}\label{sec:outline}
The usual strategy for showing that the representation ring of a split reductive group is a \(\lambda\)-ring is to use its embedding into the representation ring \(\KRep{T}\) of a maximal torus, and the fact that the latter ring is generated by line elements.
However, on the level of Grothendieck-Witt groups, the restriction to \(T\) cannot be injective:
none but the trivial character of \(T\) are symmetric, and hence the symmetric representation ring \(\GWRep{T}\) only contains one copy of the Grothendieck-Witt group of our base field. In contrast, all simple \(\OO{m}\)-modules are symmetric (see \Cref{Om-simple-modules}).

We are therefore led to look for a replacement for \(T\) with a larger supply of symmetric representations.
The candidate we choose is a semi-direct product \(T\rtimes\Z/2\) of the torus with a cyclic group of order two, which we will refer to as an ``extended torus''.  As we will see, \emph{all} representations of \(T\rtimes\Z/2\) are symmetric.

Our proof can be summarized as follows:
\begin{enumerate}[\bf Step~1.]
\item  \(\GWRep{T\rtimes\Z/2}\) is a \(\lambda\)-ring.  As all simple representations of \(T\rtimes\Z/2\) are of rank at most two, this can be checked directly in terms of the \(\lambda\)-identities. See \Cref{lambda-identities-for-low-rank-elements} and \Cref{ST-GW-special} below.

\item \(\GWRep{\SO{2n_1+1}\times\SO{2n_2+1}}\) is a \(\lambda\)-ring: it embeds into \(\GWRep{T\rtimes\Z/2}\), where \(T\) is a maximal torus in the product of special orthogonal groups. See \Cref{SOxSO-GW-special}.

\item \(\GWRep{\OO{m_1}\times\OO{m_2}}\) is a \(\lambda\)-ring: it is generated by line elements over the image of \(\GWRep{\SO{2n_1+1}\times\SO{2n_2+1}}\) for appropriate \(n_1\) and \(n_2\). See \Cref{thm:O-special}.
\end{enumerate}

\subsection{Representations of extended tori}
\newcommand{\mult}[1]{|#1|} 
The group \(\OO{2}\) is a twisted product of \(\SO{2}=\Gm\) and \(\Z/2\):
for any connected \(F\)-algebra \(A\), we have an isomorphism
\[
\begin{aligned}
  \Gm(A) \rtimes \Z/2 &\xrightarrow{\cong} \OO{2}(A)\\
  (a,x) &\mapsto \M{a & 0 \\ 0 & a^{-1}}\M{0&1\\1&0}^{x}
\end{aligned}
\]
We consider more generally semi-direct products \(T\rtimes\Z/2\), where \(T=\Gm^r\) is a split torus on which \(\Z/2\) acts by multiplicative inversion, \ie
\[
1.(a_1,\cdots, a_r) := (a_1^{-1},\cdots, a_r^{-1})
\]
for \((a_1,\dots,a_r)\in T(A)\) and \(1\in\Z/2\).
If we introduce the notation \(\mult{x}:=(-1)^x\) for \(x\in\Z/2\), we can write the action as
\[
x.(a_1,\cdots, a_r) = (a_1^{\mult{x}},\cdots, a_r^{\mult{x}}).
\]
The group structure on \(T\rtimes\Z/2\) is given by
\(
(\vec a,x)(\vec b,y) = (\vec a\cdot \vec b^{\mult{x}},x+y)
\)
in this notation.

We write \(T^*:= \Hom(T,\Gm)\) for the character group of the torus, a free abelian group of rank \(r\).  The one-dimensional \(T\)-representation corresponding to a character \(\vec\gamma\in T^*\) is denoted \(e^{\vec\gamma}\).

\begin{mdprop}[\(T\rtimes\Z/2\)-modules]
  All representation of \(T\rtimes\Z/2\) are semi-simple.
  The isomorphism classes of simple \(T\rtimes\Z/2\)-modules can be enumerated as follows:
  \begin{compactitem}
  \item[{$1$},] the trivial one-dimensional representation
  \item[{$\delta$},] the one-dimensional representation on which \(T\) acts trivially while the generator of \(\Z/2\) acts as \(-1\).
  \item[{$[e^{\vec\gamma}]$}] \(:= e^{\vec\gamma} \oplus e^{-\vec\gamma}\), for each pair of characters \(\{\vec\gamma,-\vec\gamma\}\) of \(T\) with \(\vec\gamma\neq 0\). Here, \(\Z/2\) acts by interchanging the two factors.\qed
  \end{compactitem}
\end{mdprop}
The representation \([e^{\vec 0}] := 1 \oplus 1\) with \(\Z/2\) switching the factors is isomorphic to the direct sum \(1\oplus \delta\).
For \(r=1\), \( [e^{1}] \) is the standard representation of \(\OO{2}\).

\begin{proof}
Let \(V\) be a \(T\rtimes\Z/2\)-representation.  As \(T\) is diagonalizable, the restriction of \(V\) to \(T\) decomposes into a direct sum of eigenspaces \( V_{\vec\gamma}\).  Writing \(1\) for the generator of \(\Z/2\), we find that
\(1.V_{\vec\gamma} \subset V_{-\vec\gamma}\). Thus, \(V_{\vec 0}\) is a \(T\rtimes\Z/2\)-submodule of \(V\), as is \(V_{\vec\gamma}\oplus V_{-\vec\gamma}\) for each non-zero \(\vec\gamma\).

The zero-eigenspace \(V_{\vec 0}\) may be further decomposed into copies of \(1\) and \(\delta\).  For non-zero \(\vec\gamma\), we can decompose \(V_{\vec\gamma}\) into a direct sum of copies of \(e^{\vec\gamma}\), and then \(V_{\vec\gamma}\oplus V_{\vec\gamma}\) decomposes into a direct sum of copies of \(e^{\vec\gamma} \oplus 1.e^{\vec\gamma}\cong [e^{\vec\gamma}]\).

Alternatively, we may find all simple \(T\rtimes\Z/2\)-representations by applying \Cref{reps-of-sdp:main,reps-of-sdp:lifts}.
Indeed, if we twist the \(T\)-action on \(e^{\vec \gamma}\) by \(1\in\Z/2\) (see \Cref{def:twisted-rep}), we obtain \(e^{-\vec\gamma}\), which is isomorphic to \(e^{\vec \gamma}\) if and only if \(\vec\gamma = 0\).
\end{proof}

All representations of \(T\rtimes\Z/2\) are symmetric.
For later reference, we choose a distinguished symmetry on each simple representation as follows.
\begin{mdprop}\label{tt:symmetries}
  The following symmetric representations form a basis of \mbox{\(\GWRep{T\rtimes\Z/2}\)} as a \(\GW(F)\)-module:
  \begin{compactitem}\raggedright
  \item[{$1^+$}]\( := (1, (1))\), the trivial representation equipped with the trivial symmetric form.
  \item[{$\delta^+$}]\( := (\delta, (-1))\), the representation \(\delta\) equipped with the symmetric form \((-1)\).
  \item[{$[e^{\vec\gamma}]^+$}]\( := ([e^{\vec\gamma}],\mm{0&1\\1&0}) \) for each pair \(\{\vec\gamma,-\vec\gamma\}\) with \(\gamma \neq 0\): the representation \([e^{\vec\gamma}]\) equipped with the equivariant symmetric form
    \[
    \mm{0&1\\1&0}\colon e^{\vec\gamma}\oplus e^{-\vec\gamma} \xrightarrow{\quad\cong\quad} (e^{\vec\gamma}\oplus e^{-\vec\gamma})^\vee = e^{-\vec\gamma}\oplus e^{\vec\gamma}.
    \]
  \end{compactitem}
\end{mdprop}

In short, the proposition says that we have an isomorphism of \(\GW(F)\)-modules
\[
\GWRep{T\rtimes\Z/2} = \GW(F)\vecspan{1^+, \delta^+ , [e^{\vec\gamma}]^+}_{\substack{\{\vec\gamma,-\vec\gamma\}\\\vec\gamma\neq 0}}
\]
In analogy with the notation \([e^{\vec \gamma}]^+\), we write \([e^{\vec 0}]^+\) for the representation \([e^{\vec 0}]\) equipped with the symmetric form \(\mm{0&1\\1&0}\).  There is an equivariant isometry
\[
   [e^{\vec 0}]^+ \xleftarrow{\;\cong\;} (1\oplus \delta, \mm{2 & 0 \\ 0 & -2})
\]
given by \(\mm{ 1 & 1 \\ 1 & -1}\), so \([e^{\vec 0}]^+ = (2)\cdot 1^+ + (2)\cdot\delta^+\) in \(\GWRep{T\rtimes\Z/2}\).

In order to check the \(\lambda\)-identities for \(\GWRep{T\rtimes\Z/2}\), we will need to understand tensor products and exterior powers in \(\GWRep{T\rtimes\Z/2}\). This is the subject of the next two lemmas.
\begin{lem}The tensor products of the above \(T\rtimes\Z/2\)-modules are as follows:
\begin{align}
  \tag{P1}\label{eq:tt_P1} \delta^{\otimes 2} &\cong 1\\
  \tag{P2}\label{eq:tt_P2} \delta\otimes [e^{\vec\gamma}] &\cong [e^{\vec\gamma}]\\
  \tag{P3}\label{eq:tt_P3} [e^{\vec\gamma}] \otimes [e^{\vec\kappa}] &\cong [e^{\vec\gamma + \vec\kappa}] \oplus [e^{\vec\gamma - \vec\kappa}]
\end{align}
Here, \(\vec\gamma\) and \(\vec \kappa\) are arbitrary characters of \(T\) (possibly zero).
\end{lem}
\begin{proof}
The first isomorphism is obvious.
For the other two isomorphisms, we describe \([e^{\vec\gamma}]\)
as the two-dimensional representation with basis \(v_1,v_{-1}\) and action
\[
(\vec a,x).v_{\mu} := v_{\mult{x}\mu}\otimes a_1^{\mult{x}\mu\gamma_1}\cdots a_r^{\mult{x}\mu\gamma_r}.
\]
In this notation, an explicit isomorphism \([e^{\vec\gamma}]\cong[e^{-\vec\gamma}]\) is given by
\begin{equation}\label{eq:tt_R}
  \tag{R}
  \M{0&1\\1&0}\colon
  \begin{aligned}
    [e^{\vec\gamma}] &\xrightarrow{\cong} [e^{-\vec\gamma}]\\
    v_{\mu} &\mapsto v_{-\mu}
  \end{aligned}
\end{equation}
The second two isomorphisms of the lemma can be described as follows:
\begin{align}
  \tag{P2}
  \M{1&0\\0&-1}\colon&
  \begin{aligned}\delta\otimes [e^{\vec\gamma}] &\xrightarrow{\cong} [e^{\vec\gamma}]\\
    \notag v_\mu &\mapsto \mu v_\mu
  \end{aligned}\\
  \notag\\
  \tag{P3}
  \phantom{\M{1&0\\0&-1}\colon}&
  \begin{aligned}
    [e^{\vec\gamma}] \otimes [e^{\vec\kappa}] &\xrightarrow{\cong} [e^{\vec\gamma + \vec\kappa}] \oplus [e^{\vec\gamma - \vec\kappa}]\\
    v_\mu  \otimes v_{\mu\phantom{-}}    &\mapsto \quad (v_\mu,\;0\;)\\
    v_\mu \otimes v_{-\mu}              &\mapsto \quad (\;0\;, v_\mu)
  \end{aligned}
\end{align}
\end{proof}

\begin{lem}\label{lem:tt:isometries}
  With respect to the symmetric forms chosen in \Cref{tt:symmetries}, the isomorphisms \eqref{eq:tt_P1}--\eqref{eq:tt_P3} and \eqref{eq:tt_R} are isometries:
  \begin{align}
    \tag{P1${}^+$}\label{eq:tt_P1+} (\delta^+)^{\otimes 2} &\cong 1^+\\
    \tag{P2${}^+$}\label{eq:tt_P2+} \delta^+\otimes [e^{\vec\gamma}]^+ &\cong [e^{\vec\gamma}]^+\\
    \tag{P3${}^+$}\label{eq:tt_P3+} [e^{\vec\gamma}]^+ \otimes [e^{\vec\kappa}]^+ &\cong [e^{\vec\gamma + \vec\kappa}]^+ \perp [e^{\vec\gamma - \vec\kappa}]^+\\
\tag{R${}^+$}\label{eq:tt_R+} [e^{\vec\gamma}]^+ & \cong [e^{-\vec\gamma}]^+
\intertext{Moreover, for each character \(\vec\gamma\) of \(T\), we have an isometry}
    \tag{L${}^+$}\label{eq:tt_L+} \Lambda^2([e^{\vec\gamma}]^+) &\cong \delta^+.
   \end{align}
\end{lem}
\begin{proof}
These are claims are about symmetries on vector spaces, which can be checked in a basis. \eqref{eq:tt_P1+} is clear. \eqref{eq:tt_R+} and \eqref{eq:tt_P2+} and boil down to the following two identities:
\begin{alignat*}{5}
\mm{0 & 1 \\ 1 & 0}\;
& = & {^t\!\mm{0 & 1 \\ 1 & 0}} &\mm{0 & 1 \\ 1 & 0} \mm{0 & 1 \\ 1 & 0} \\
\underbrace{\mm{0 & -1 \\ -1 & 0}}_{\mathclap{\substack{\text{symmetry}\\\text{on LHS}}}}
& = & \;{^t\!\mm{1 & 0 \\ 0 & -1}}&\underbrace{\mm{0 & 1 \\ 1 & 0}}_{\mathclap{\substack{\text{symmetry}\\\text{on RHS}}}} \mm{1 & 0 \\ 0 & -1}
\end{alignat*}
For \eqref{eq:tt_P3+}, we choose the following bases:
\begin{alignat*}{12}
  &v_1\otimes v_1,\quad &&v_1\otimes v_{-1}, \quad &&v_{-1}\otimes v_1,\quad &&v_{-1}\otimes v_{-1} &&\in & [e^{\vec\gamma}]&\otimes [e^{\vec\kappa}]\\
  &(v_1, 0),       &&(v_{-1}, 0),       &&(0, v_1),             &&(0, v_{-1})          &&\in & [e^{\vec\gamma + \vec\kappa}] &\oplus [e^{\vec\gamma - \vec\kappa}]
\end{alignat*}
Then the symmetry on \([e^{\vec\gamma}]^+\otimes [e^{\vec\kappa}]^+\), the symmetry on \([e^{\vec\gamma + \vec\kappa}] \perp [e^{\vec\gamma - \vec\kappa}]\) and the isomorphism \eqref{eq:tt_P3} are represented by the following three matrices \(Q_L\), \(Q_R\) and \(I\), which do indeed satisfy the required identity \(Q_L = {^t I}Q_R I\):
\[Q_L = \mm{0 & 0 & 0 & 1 \\ 0 & 0 & 1 & 0 \\ 0 & 1 & 0 & 0 \\ 1 & 0 & 0 & 0},\quad
  Q_R = \mm{0 & 1 & 0 & 0 \\ 1 & 0 & 0 & 0 \\ 0 & 0 & 0 & 1 \\ 0 & 0 & 1 & 0},\quad
  I = \mm{1 & 0 & 0 & 0 \\ 0 & 0 & 0 & 1 \\ 0 & 1 & 0 & 0 \\ 0 & 0 & 1 & 0}
\]
Finally, \(\Lambda^2([e^{\vec\gamma}]^+)\) has basis \(v_1\wedge v_{-1}\), on which \(1\in \Z/2\) acts as as \(-1\), and symmetry \(\det\mm{0 & 1 \\ 1 & 0} = -1\).
So \(\Lambda^2([e^{\vec\gamma}]\) is isometric to \(\delta^+ = (\delta, -1)\), as claimed.
\end{proof}

\begin{prop}\label{ST-GW-special}
  \(\GWRep{T\rtimes\Z/2}\) is a \(\lambda\)-ring.
\end{prop}
\begin{proof}
  By \cref{lem:special-via-line-generators}, it suffices to check the \(\lambda\)-identities for the symmetric representations listed in \Cref{tt:symmetries}.  As all of these are of rank one or two, these identities boil down to the equations given in part~\eqref{lambda-identities-for-rank-two} of \Cref{lambda-identities-for-low-rank-elements}.  By part~\eqref{lambda-identities-vs-line-special} of the same \namecref{lambda-identities-for-low-rank-elements} and the fact that \(\GWRep{T\rtimes\Z/2}\) is line-special (\Cref{lem:line-special}), we can concentrate on the case \(x=[e^{\vec\gamma}]^+\), \(y=[e^{\vec\kappa}]^+\).  So it suffices to check the following three identities in \(\GWRep{T\rtimes\Z/2}\):
  \[
  \begin{aligned}
   (1)\quad \lambda^2([e^{\vec\gamma}]^+ \cdot [e^{\vec\kappa}]^+)
   &= \left(([e^{\vec\gamma}]^+)^2 - 2\lambda^2[e^{\vec\gamma}]^+\right)\cdot\lambda^2[e^{\vec\kappa}]^+\\
   &\quad+ \left(([e^{\vec\kappa}]^+)^2-2\lambda^2[e^{\vec\kappa}]^+\right)\cdot\lambda^2 [e^{\vec\gamma}]^+ \\
   &\quad+ 2\lambda^2[e^{\vec\gamma}]^+\cdot\lambda^2 [e^{\vec\kappa}]^+ \\
   (2)\quad \lambda^3([e^{\vec\gamma}]^+ \cdot [e^{\vec\kappa}]^+) &= [e^{\vec\gamma}]^+ \cdot [e^{\vec\kappa}]^+\cdot\lambda^2[e^{\vec\gamma}]^+\cdot\lambda^2 [e^{\vec\kappa}]^+\\
   (3)\quad \lambda^4([e^{\vec\gamma}]^+ \cdot [e^{\vec\kappa}]^+) &= (\lambda^2[e^{\vec\gamma}]^+)^2\cdot(\lambda^2[e^{\vec\kappa}]^+)^2
  \end{aligned}
  \]
Using the lemmas above, we see that both sides of $(3)$ equate to \(1^+\) and that both sides of $(2)$ equate to \([e^{\vec\kappa + \vec\gamma}]^+ + [e^{\vec\kappa - \vec\gamma}]^+\).  Equation~$(1)$ simplifies to
\[
   [e^{2\vec \gamma}]^+ + [e^{2\vec \kappa}]^+ + 2\delta^+  = [e^{2\vec \gamma}]^+ + [e^{2 \vec \kappa}]^+ + 2[e^{\vec 0}]^+ - 2\cdot 1^+.
\]
Thus, it suffices to show that
\[
2\cdot 1^+  + 2\delta^+ = 2[e^{\vec 0}]^+
\]
in \(\GWRep{T\rtimes\Z/2}\). As indicated below \Cref{tt:symmetries}, \([e^{\vec 0}]^+ = (2)\cdot 1^+ + (2) \cdot \delta^+\), so we can rewrite this equation as
\[
\mm{1 & 0 \\ 0 & 1}\cdot 1^+ + \mm{1 & 0 \\ 0 & 1}\cdot\delta^+ = \mm{2 & 0 \\ 0 & 2}\cdot 1^+ + \mm{2 & 0 \\ 0 & 2}\cdot \delta^+.
\]
This equation does indeed hold in \(\GWRep{T\rtimes\Z/2}\) since the non-degenerate symmetric forms \(\mm{1 & 0 \\ 0 & 1}\) and \(\mm{2 & 0 \\ 0 & 2}\) are isometric over any field of characteristic not two via the isometry \(\mm{1 & 1 \\ 1 & -1}\).
\end{proof}

\begin{rem*}
  When the ground field \(k\) contains a square root of \(2\), the representation \([e^{\vec 0}]^+\) is isometric to \(1^+ \perp \delta^+\).  In this case,  \Cref{ST-GW-special} can alternatively be proved without explicitly checking the \(\lambda\)-identities.  Namely, it then follows from \Cref{lem:tt:isometries} that the subgroup
  \[
  \GWRep{T\rtimes\Z/2}^+ := \Z\vecspan{1^+, \delta^+ , [e^{\vec\gamma}]^+}_{\substack{\{\vec\gamma,-\vec\gamma\}\\\vec\gamma\neq 0}}
  \]
  is closed under products and exterior powers and is thus a sub-pre-\(\lambda\)-ring of \(\GWRep{T\rtimes\Z/2}\).
  This subring maps isomorphically to the \(\lambda\)-ring \(\KRep{T\rtimes\Z/2}\) via the forgetful map and is thus itself a \(\lambda\)-ring. On the other hand, \(\GWRep{T\rtimes\Z/2}\) is generated over \(\GWRep{T\rtimes\Z/2}^+\) by line elements, so we can conclude via \cref{lem:special-via-line-generators}.
\end{rem*}

\subsection{Representations of split special orthogonal groups}
Let \(\SO{2n}\) and \(\SO{2n+1}\) be the split orthogonal groups defined by the standard symmetric forms \eqref{eq:standard-symmetric-forms} (\Cref{sec:symreps}).
The aim of this section is to show that the symmetric representation rings of these groups are \(\lambda\)-rings.
We begin by explaining our notation and conventions.

First, note that the two groups share a maximal torus \(T\subset \SO{2n}\subset \SO{2n+1}\) which we can write as
\[
\begin{aligned}
T(A) &= \left\{\diag{\mm{a_1 & 0\\ 0 & a_1^{-1}},\dots, \mm{a_n & 0 \\ 0 & a_n^{-1}}} \quad \;\middle\vert\; a_i\in A^\times\right\}\\
\quad\text{ respectively }\quad
T(A) &= \left\{\diag{\mm{a_1 & 0 \\ 0 & a_1^{-1}},\dots, \mm{a_n & 0 \\ 0 & a_n^{-1}}, 1} \;\middle\vert\; a_i\in A^\times \right\}.
\end{aligned}
\]
Let
\[
  \vec \gamma = (\gamma_1,\dots,\gamma_n) \in T^*
\]
denote the character \(T\to\Gm\) that sends an element of \(T(A)\) of the form indicated to the product \(a_1^{\gamma_1}\cdot\cdots\cdot a_n^{\gamma_n}\).
Then with respect to the usual choices, the dominant characters for \(\SO{m}\) are described by the conditions
\[\begin{cases}
  \gamma_1 \geq \gamma_2 \geq \cdots \geq \gamma_{n-1} \geq \gamma_n \geq 0 & (m = 2n+1)\\
  \gamma_1 \geq \gamma_2 \geq \cdots \geq \gamma_{n-1} \geq \abs{\gamma_n}  & (m = 2n)\\
\end{cases}\]
We refer to these as \(2n\)- and \((2n+1)\)-dominant, respectively.
Given any \(\vec \gamma\in T^*\), we define
\[
\vec \gamma^- := (\gamma_1,\dots,\gamma_{n-1},-\gamma_n).
\]

\begin{lem}\label{lem:dominant}\hfill\\
If \(\gamma_n\geq 0\), then \(\vec\gamma\) is \(2n\)-dominant if and only if it is \((2n+1)\)-dominant.\\
If \(\gamma_n\leq 0\), then \(\vec\gamma\) is \(2n\)-dominant if and only if \(\vec \gamma^-\) is \((2n+1)\)-dominant.
\qed
\end{lem}

The usual partial order on \(T^*\) is given by
\[
\begin{aligned}
  \vec \gamma' \preceq \vec \gamma
  &\phantom{:}\Leftrightarrow
  \begin{cases}
    \ltext{8cm}{\((\vec \gamma', \vec \gamma)\) satisfy conditions $(C_1)$--$(C_n)$ } & (m=2n+1) \\
    \ltext{8cm}{\((\vec \gamma', \vec \gamma)\) satisfy conditions $(C_1)$--$(C_n)$ and $(C_n^-)$ } & (m=2n) \\
  \end{cases}
\end{aligned}
\]
where the conditions are:
\[
\begin{aligned}
  &(C_i)  & \gamma_1' + \cdots + \gamma_i' \phantom{-\gamma_n'} & \leq \phantom{-\gamma_n'} \gamma_1 + \cdots + \gamma_i\\
  &(C_n^-)  & \gamma_1' + \cdots + \gamma_{n-1}' - \gamma_n' &\leq \gamma_1 + \cdots + \gamma_{n-1} - \gamma_n\\
\end{aligned}
\]
We say that a weight \(\vec \gamma'\) is smaller than \(\vec \gamma\) and write \(\vec \gamma'\prec \vec \gamma\) when \(\vec\gamma' \preceq \vec\gamma \) and \( \vec\gamma' \neq \vec \gamma\). When it is not clear from the context, we say \(2n\)- or \((2n+1)\)-smaller to clarify which of the two orderings we are using. Note that any character \(\vec\gamma'\) which is \(2n\)-smaller than \(\vec\gamma\) is a fortiori \((2n+1)\)-smaller.
We will eventually need the following observation.
\begin{lem}\label{lem:minus-ordering}
    Let \(\vec\gamma\) and  \(\vec\gamma'\) be \((2n+1)\)-dominant.
    If \(\vec\gamma'\) is \((2n+1)\)-smaller than \(\vec\gamma^-\), then it is \(2n\)-smaller than \(\vec\gamma\).\qed
\end{lem}

Given an \(m\)-dominant weight \(\vec\gamma\), we write \(V_{\vec\gamma}\) for the simple \(\SO{m}\)-module of highest weight~\(\vec\gamma\).
\begin{lem}\label{lem:SO-symmetries}\hfill\\
When \(m\) is odd or a multiple of four, all simple representations of \(\SO{m}\) are symmetric.\\
When \(m = 2n\) with \(n\) odd, all simple representations \(V_{\vec\gamma}\) with \(\gamma_n = 0\) are symmetric, while for \(\gamma_n \neq 0\) we have \(V_{\vec\gamma}^{\dual} \cong V_{\vec\gamma^-}\).
\end{lem}

One way to prove this lemma is to consider the restriction from \(\SO{m}\) to an appropriate extended torus \(T\rtimes\Z/2\). We can define such tori by specifying generators \(\tau\) of \(\Z/2\):
\begin{compactitem}
\item For \(m = 2n+1\), let
  \[
  \tau := \mm{
    0&1& & & & & \\
    1&0& & & & & \\
    & &\dots&&&&\\
    & & &\dots&&&\\
    & & & &0&1& \\
    & & & &1&0& \\
    & & & & & &(-1)^n\\
  }.
  \]
  The conjugation action of \(\tau\) on \(T\) is multiplicative inversion, so the subgroup generated by \(\tau\) and \(T\) is indeed an extended torus \(T\rtimes\Z/2\).
\item For \(m=2n\) with \(n\) even, let
  \[
  \tau := \mm{
    0&1& & & & \\
    1&0& & & & \\
    & &\dots&&&\\
    & & &\dots&&\\
    & & & &0&1\\
    & & & &1&0}.
  \]
  Then, as in the previous case, the subgroup generated by \(\tau\) and \(T\) is an extended torus \(T\rtimes\Z/2\).
\item For \(m=2n\) with \(n\) odd, let
  \[
  \tau := \mm{
    0&1& & & & & & \\
    1&0& & & & & & \\
    & &\dots&&&&& \\
    & & &\dots&&&& \\
    & & & &0&1& & \\
    & & & &1&0& & \\
    & & & & & &1& \\
    & & & & & & &1
  }.
  \]
  Write \(T\) as \(T'\times\Gm\). The element \(\tau\) acts by multiplicative inversion on \(T'\) and trivially on \(\Gm\), so the subgroup generated by \(\tau\) and \(T\) has the form \((T'\rtimes\Z/2)\times\Gm\).
\end{compactitem}

\newcommand{\subsymT}[1]{ST_{#1}}
We thus have closed subgroups
\begin{alignat*}{7}
  &\subsymT{2n+1} &&:= T\rtimes\Z/2 &&\hookrightarrow \SO{2n+1}\\
  &\subsymT{2n} &&:= T\rtimes\Z/2 &&\hookrightarrow \SO{2n} \quad\text{ for even } n\\
  &\subsymT{2n} &&:= (T'\rtimes\Z/2)\times \Gm &&\hookrightarrow \SO{2n} \quad\text{ for odd } n
\end{alignat*}

\begin{proof}[Proof of \Cref{lem:SO-symmetries}]
By \cite{Jantzen}*{Cor.~II.2.5}, \(V_{\vec\gamma}^\dual \cong V_{-w_0\vec\gamma}\).

When \(m\) is odd or divisible by four, \(-w_0\) is the identity, so all simple \(\SO{m}\)-modules are self-dual.
The restriction of \(V_{\vec\gamma}\) to \(\subsymT{m}\) contains the symmetric simple \(\subsymT{m}\)-module \([e^{\vec\gamma}] = e^{\vec\gamma}\oplus e^{\vec\gamma-}\) as a direct summand with multiplicity one, so the duality on \(V_{\vec\gamma}\) must be symmetric.

When \(m=2n\) with \(n\) odd, \(-w_0\) acts as \(\vec\gamma\mapsto \vec\gamma^-\).
Thus, for a dominant character \(\vec\gamma\) the \(\SO{m}\)-module \(V_{\vec\gamma}\) is dual to \(V_{\vec\gamma^-}\). In the case \(\gamma_n= 0\) we can argue as before.
\end{proof}

The restriction \(\KRep{\SO{m}} \to \KRep{\subsymT{m}}\) is a monomorphism since its composition with the further restriction to \(T\) is.
We claim that the same is true for Grothendieck-Witt groups.

\begin{mdprop}\label{SO-GW-special}
  The morphism induced by restriction
  \begin{equation*}
  \GWRep{\SO{m}} \xrightarrow{i^*} \GWRep{\subsymT{m}}
  \end{equation*}
  is injective.
  In particular, \(\GWRep{\SO{m}}\) is a  \(\lambda\)-ring.
\end{mdprop}
At this point, we only give a proof of this proposition in the case when \(m\) is odd.
The cases when \(m\) is even can be dealt with similarly, but we will not need them for our further analysis.

We will, however, need a generalization to products of special orthogonal groups of odd ranks.
If \(\tau_1\in\SO{2n_1+1}\) and \(\tau_2\in\SO{2n_2+1}\) are the elements of order two defined above, then \((\tau_1,\tau_2)\) is an element of order two in \(\SO{2n_1+1}\times\SO{2n_2+1}\).  Together with the maximal tori \(T_1\) and \(T_2\) of \(\SO{2n_1+1}\) and \(\SO{2n_2+1}\), it generates an extended torus \((T_1\times T_2)\rtimes \Z/2\).

\begin{mdprop}\label{SOxSO-GW-special}
  The morphism induced by restriction
  \[
  \GWRep{\SO{2n_1+1}\times\SO{2n_2+1}} \to \GWRep{(T_1\times T_2)\rtimes \Z/2}
  \]
  is injective.
  In particular, \(\GWRep{\SO{2n_1+1}\times\SO{2n_2+1}}\) is a \(\lambda\)-ring.
\end{mdprop}

\begin{proof}[Proof of \Cref{SO-GW-special} for odd \(m\)]
The main point is to show that the restriction is a monomorphism. Once we know this, the claim that \(\GWRep{\SO{m}}\) is a \(\lambda\)-ring follows from \Cref{ST-GW-special}.

As a preliminary exercise, consider the restriction \(i^*\colon \KRep{\SO{m}}\to \KRep{T}\). One way to argue that this morphism is injective is as follows:
  \begin{quote}
    The isomorphism classes of the simple \(\SO{m}\)-modules \(V_{\vec\gamma}\) for the different dominant weights \(\vec\gamma\) form a \(\Z\)-basis of \(\KRep{\SO{m}}\). Moreover, for dominant weights \(\vec\gamma\) and \(\vec\mu\),
    \begin{compactitem}[-]
    \item the coefficient of \(e^{\vec\gamma}\) in \(i^*V_{\vec\mu}\) is non-zero only if \(\vec\mu\succeq\vec\gamma\),
    \item the coefficient of \(e^{\vec\gamma}\) in \(i^*V_{\vec\gamma}\) is \(1\).
    \end{compactitem}
    Suppose \(e := \sum_{\vec\gamma}n_{\vec\gamma}V_{\vec\gamma}\) is an element of \(\KRep{\SO{m}}\) that restricts to zero. Let \(\vec\gamma\) be maximal among all dominant weights for which \(n_{\vec\gamma}\neq 0\). Then the coefficient of \(e^{\vec\gamma}\) in \(i^*e\) is precisely \(n_{\vec\gamma}\). So we find that \(n_{\vec\gamma}\) must be zero, a contradiction.
  \end{quote}
  Essentially the same argument can be applied to the restriction \(\GWRep{\SO{m}}\to\GWRep{\subsymT{m}}\).
  As we have seen, all simple \(\SO{m}\)-modules are symmetric.
    Choose a symmetry \(\theta_{\vec\gamma}\) on \(V_{\vec\gamma}\) for each dominant \(\vec\gamma\), so that
    \[
\begin{aligned}
    \GWRep{\SO{m}} &= \GW(F)\vecspan{(V_{\vec\gamma}, \theta_{\vec\gamma})}_{\vec\gamma \text{ dominant}}\\
    \GWRep{\subsymT{m}} &= \GW(F)\vecspan{1^+,\delta^+,[e^{\vec\gamma}]^+}_{\{\vec\gamma,-\vec\gamma\} \text{ with } \vec \gamma\neq 0}.
    \end{aligned}
\]
    For dominant weights \(\vec\gamma\), \(\vec\mu\),
    \begin{compactitem}[-]
    \item the coefficient of \([e^{\vec\gamma}]^+\) in \(i^*(V_{\vec\mu},\theta_{\vec\mu})\) is non-zero only if \(\vec\mu\succeq\vec\gamma\),
    \item the coefficient of \([e^{\vec\gamma}]^+\) in \(i^*(V_{\vec\gamma},\theta_{\vec\gamma})\) is a non-degenerate rank-one symmetric form \((\alpha_{\vec\gamma})\).
    \end{compactitem}
    Indeed, this can be checked by restricting to \(\KRep{T_m}\).

    Now suppose \(e := \sum_{\vec\gamma}\phi_{\vec\gamma}(V_{\vec\gamma},\theta_{\vec\gamma})\) with \(\phi_{\vec \gamma} \in \GW(F)\) is an element of \(\GWRep{\SO{m}}\) that restricts to zero in \(\GWRep{\subsymT{m}}\).
 Let \(\vec\gamma\) be maximal among all dominant weights for which \(\phi_{\vec\gamma}\neq 0\) in \(\GW(F)\). Then the coefficient of \(e^{\vec\gamma}\) in \(i^*e\) is \(\phi_{\vec\gamma}\cdot(\alpha_{\vec\gamma})\). As \((\alpha_{\vec\gamma})\) is invertible, we find that \(\phi_{\vec\gamma} = 0\), a contradiction.
\end{proof}

\begin{proof}[Proof of \Cref{SOxSO-GW-special}]
  A character \(\vec\gamma = (\vec\gamma_1,\vec\gamma_2)\in (T_1\times T_2)^*\) is dominant for \(\SO{2n_1+1}\times\SO{2n_2+1}\) if and only if each \(\vec\gamma_i\) is dominant for \(\SO{2n_i+1}\).
  The corresponding simple \(\SO{2n_1+1}\times\SO{2n_2+1}\)-modules are of the form
  \(
  V_{\vec\gamma} = V_{\vec\gamma_1}\otimes V_{\vec\gamma_2}.
  \)
  In particular, all simple modules are again symmetric.
  Using the ordering
  \[
  (\vec\mu_1,\vec \mu_2) \succeq (\vec\gamma_1,\vec\gamma_2) \quad :\Leftrightarrow \quad \vec\mu_1 \succeq \vec \gamma_1 \text{ and } \vec\mu_2\succeq \vec\gamma_2,
  \]
  we can argue exactly as before.
\end{proof}

\subsection{Representations of split orthogonal groups}
In this section, we finally deduce that the symmetric representation ring of the split orthogonal group \(\OO{m}\) is a \(\lambda\)-ring.

We take \(\OO{m}\) to be defined by one of the standard symmetric forms \eqref{eq:standard-symmetric-forms} (\Cref{sec:symreps}) and write \(T\) for the maximal torus of \(\SO{m}\) defined in the previous section.
When \(m\) is odd, the orthogonal group decomposes as a direct product, while for even \(m\), we only have a semi-direct product decomposition:
\[
\begin{aligned}
  \OO{2n+1} & \cong \SO{2n+1}\times\Z/2 &&\text{with \(\Z/2\) generated by \(-1\)}\\
  \OO{2n\phantom{+1}} &\cong \SO{2n\phantom{+1}}\rtimes\Z/2 &&\text{with \(\Z/2\) generated by \(\diag{1,\dots 1, \mm{0&1\\1&0}}\)}
\end{aligned}
\]
We thus obtain the following well-known description of the simple \(\OO{m}\)-modules.

\begin{mdprop}\label{Om-simple-modules}\hfill
  \begin{compactenum}[(a)]
  \item Each simple \(\SO{2n+1}\)-module \(V_{\vec\gamma}\) lifts to two distinct simple \(\OO{2n+1}\)-modules
    \(\hat V_{\vec\gamma}\) and \(\hat V_{\vec\gamma}\otimes\delta\),
    where \(\delta\) denotes the non-trivial character of \(\Z/2\).
    All simple \(\OO{2n+1}\)-modules arise in this way.
    \[
    \left.\begin{aligned}
        &\hat V_{\vec \gamma}\\
        &\hat V_{\vec \gamma}\otimes\delta
      \end{aligned}\right\}\xmapsto{\text{res}} V_{\vec\gamma}
    \]
  \item Likewise, each simple \(\SO{2n}\)-module \(V_{\vec\gamma}\) with \(\gamma_n=0\) lifts to two distinct simple \(\OO{2n}\)-modules
    \(  \hat V_{\vec\gamma}\) and \(\hat V_{\vec\gamma}\otimes\delta  \).
    For each dominant \(\vec\gamma\) with \(\gamma_n>0\) there is a simple \(\OO{2n}\)-module
    \( \hat V_{\vec \gamma}  \)
    lifting \(V_{\vec\gamma}\oplus V_{\vec\gamma^-}\); in this case, \(V_{\vec\gamma}\) and \(V_{\vec\gamma}\otimes\delta\) are isomorphic.
    Every simple \(\OO{2n}\)-module arises in one of these two ways.
    \begin{alignat*}{7}
      &\gamma_n= 0\colon\quad &&\left.\begin{aligned}
          &\hat V_{\vec \gamma}\\
          &\hat V_{\vec \gamma}\otimes\delta
        \end{aligned}\right\}&&\xmapsto{\text{res}} V_{\vec\gamma}\\
      &\gamma_n > 0\colon\quad &&\;\hat V_{\vec\gamma} &&\xmapsto{\phantom{\text{res}}} V_{\vec\gamma} \oplus V_{\vec\gamma^-}
    \end{alignat*}
  \item Every simple \(\OO{m}\)-module is symmetric, for any \(m\), and its endomorphism ring is equal to the base field \(F\).
  \end{compactenum}
\end{mdprop}

\begin{proof}\hfill\\
{\it (a) } is immediate from \Cref{reps-of-products}, and {\it (b) } follows from \Cref{reps-of-sdp:semi-simple}:  the automorphism \(\alpha\) defined by conjugation with \(\tau\) already stabilizes our chosen split maximal torus \(T\) in this case and the induced action on \(T^*\) is given by \(\alpha^*(\vec\gamma) = \vec\gamma_-\).

{\it (c)} The claim concerning the endomorphism ring follows from the corresponding fact for simple \(\SO{m}\)-modules and \Cref{reps-of-sdp:endo}.   It remains to show that the simple \(\OO{m}\)-modules are symmetric.   For odd \(m\), this is clear. So we concentrate on the case \(m=2n\).

When \(\gamma_n\neq 0\), the module \(\hat V_{\vec\gamma}\) is necessarily self-dual because it is the only simple \(\OO{2n}\)-module restricting to the self-dual \(\SO{2n}\)-module \(V_{\vec\gamma}\oplus V_{\vec\gamma^-}\).  To see that it is symmetric, we consider instead the restriction to \(T\rtimes\Z/2\), where \(\Z/2\) acts non-trivially only on the last coordinate of \(T\) (see the beginning of the proof): this restriction contains the simple symmetric \(T\rtimes\Z/2\)-module
\[
e^{(\gamma_1,\dots,\gamma_{n-1})} \otimes [e^{\gamma_n}] = e^{\vec\gamma}\oplus e^{\vec \gamma^-}
\]
as a direct summand with multiplicity one.

When  \(\gamma_n = 0\), we know that \(\hat V_{\vec\gamma}\) is either self-dual or dual to \(\hat V_{\vec\gamma}\otimes\delta\).
In order to exclude the second possibility, recall that by our construction \(\Z/2\) acts trivially on \(e^{\vec \gamma}\subset \hat V_{\vec\gamma}\).
It follows that \(\Z/2\) acts trivially on the two-dimensional subspace \(e^{\vec \gamma}\oplus e^{-\vec \gamma}\):
\begin{hilfslem*} When \(\gamma_n = 0\), \(\Z/2\) acts trivially on \(e^{\vec\gamma}\oplus e^{-\vec\gamma}\subset \hat V^{\vec\gamma}\).
\end{hilfslem*}
\proof[Proof of the lemma]
  Consider the element
  \[
  \tau := \begin{cases}
    \diag{\mm{0&1\\1&0}, \dots\dots\dots, \mm{0&1\\1&0}} &\in \SO{2n}(F)  \quad\text{ if \(n\) is even} \\
    \diag{\mm{0&1\\1&0}, \dots, \mm{0&1\\1&0}, 1,1} &\in \SO{2n}(F)  \quad\text{ if \(n\) is odd}
      \end{cases}
  \]
  This element commutes with the generator  \(1 = \diag{1,\dots 1, \mm{0&1\\1&0}}\) of \(\Z/2\) and maps \(e^{\vec\gamma}\) to \(e^{(-\gamma_1,\dots,-\gamma_{n-1},\gamma_n)}\) = \(e^{-\vec\gamma}\) (\cf \Cref{reps-twist}).
\hilfsqed

The lemma implies that \(\Z/2\) also acts trivially on the subspace \(e^{\vec\gamma} \oplus e^{-\vec\gamma} \subset \hat V_{\vec\gamma}^\dual\). So we must have \(\hat V_{\vec\gamma}^\dual\cong\hat V_{\vec\gamma}\).
Finally, as \(\hat V_{\vec\gamma}\) restricts to a simple symmetric \(\SO{2n}\)-module, the duality on \(\hat V_{\vec\gamma}\) must be symmetric.
\end{proof}

\begin{cor}\label{thm:O-special}
  \(\GWRep{\OO{m}}\) and \(\GWRep{\OO{m_1}\times\OO{m_2}}\) are \(\lambda\)-rings for all \(m\), \(m_1\), \(m_2\).
\end{cor}
\begin{proof}\mbox{}
  \begin{asparaenum}[(1)]
  \item For the orthogonal group \(\OO{2n+1}\) of odd rank, the inclusion \(\SO{2n+1}\hookrightarrow\OO{2n+1}\) is split by the projection \(q_{2n+1}\colon\OO{2n+1}\twoheadrightarrow\SO{2n+1}\). Consider the induced inclusion of pre-\(\lambda\)-rings
    \[
\begin{aligned}
      \GWRep{\SO{2n+1}} &\xhookrightarrow{q_{2n+1}^*} \GWRep{\OO{2n+1}}.
    \end{aligned}
\]
    By \Cref{SO-GW-special}, \(\GWRep{\SO{2n+1}}\) is a \(\lambda\)-ring.
    The description of the simple \(\OO{2n+1}\)-modules in \Cref{Om-simple-modules} implies that \(\GWRep{\OO{2n+1}}\) is generated as a \(\GWRep{\SO{2n+1}}\)-module by \(1^+\) and \(\delta^+\).
   We may therefore conclude via \cref{lem:special-via-line-generators}.

  \item For the orthogonal group \(\OO{2n}\) of even rank we consider its embedding into \(\SO{2n+1}\):
    \begin{equation}\label{eq:O2n-inclusion}
      \xymatrix@R=6pt{
        {\SO{2n}} \ar@{^{(}.>}[r] \ar@{_{(}.>}[d] & {\SO{2n+1}}\\
        {\OO{2n}} \ar@{_{(}->}[ru]_{q_{2n}}
        }
      \end{equation}
        The image of the induced ring morphisms
    \[
\begin{aligned}
      \GWRep{\SO{2n+1}}&\xrightarrow{q_{2n}^*}\GWRep{\OO{2n}}
    \end{aligned}
\]
    is a \(\GW(F)\)-submodule of \(\GWRep{\OO{2n}}\) which is a sub-pre-\(\lambda\)-ring and, by \Cref{SO-GW-special}, in fact a \(\lambda\)-ring. We claim that, as above, \(\GWRep{\OO{2n}}\) is generated over this submodule by the line elements \(1^+\) and \(\delta^+\):
    \begin{align}\label{eq:GW-O2n}
      \im(q_{2n}^*)\vecspan{1^+, \delta^+} = \GWRep{\OO{2n}}
    \end{align}
    Then the claim that \(\GWRep{\OO{2n}}\) is a \(\lambda\)-ring follows as for \(\OO{2n+1}\).

    To prove \eqref{eq:GW-O2n}, we analyse how the simple \(\SO{2n+1}\)-modules restrict to \(\OO{2n}\). As our chosen maximal torus \(T\) is contained in all three groups in diagram~\eqref{eq:O2n-inclusion}, we can do so by comparing weights.
    So let \(\gamma\) be a \((2n+1)\)-dominant character in \(T^*\).
In the simple \(\SO{2n+1}\)-module \(V_{\vec \gamma}\), the weight \(\gamma\) appears with multiplicity one, while all other weights that occur are \((2n+1)\)-smaller than \(\gamma\).\\

    \begin{hilfslem*}[Weights of \(\OO{2n}\)-modules]
      Let \(\vec \gamma\) be \((2n+1)\) dominant, and let \(\hat V_{\vec\gamma}\) be the corresponding \(\OO{2n}\)-module as in \Cref{Om-simple-modules}.
      Then \(e^{\vec\gamma}\) appears with multiplicity one in the restrictions of \(\hat V_{\vec\gamma}\)  and \(\hat V_{\vec\gamma}\otimes\delta\) to \(T\).  All other \((2n+1)\)-dominant weights \(\vec\mu\) for which \(e^{\vec\mu}\) occurs in these restrictions are \((2n+1)\)-smaller than \(\vec \gamma\).
    \end{hilfslem*}
    In fact, as the following proof shows, these weights \(\vec\mu\) are even \(2n\)-smaller than \(\vec\gamma\).

    \proof[Proof of the Lemma]
      If \(\gamma_n = 0\), we see from the restriction to \(\SO{2n}\) that \(\vec\gamma\) occurs with multiplicity one, and that \emph{all} other weights that occur are \(2n\)-smaller than \(\vec\gamma\).

      If \(\gamma_n > 0\), the restriction to \(\SO{2n}\) shows that \(\vec\gamma\) and \(\vec\gamma^-\) both occur with multiplicity one, and that all other weights that occur are either \(2n\)-smaller than one or \(2n\)-smaller than the other. Of these, the \emph{\((2n+1)\)-dominant} weights will necessarily be \(2n\)-smaller than \(\vec\gamma\), by \Cref{lem:minus-ordering}.
    \hilfsqed

    Now consider the \(\SO{2n+1}\)-representation \(V_{\vec\gamma}\) and its restriction to \(\OO{2n}\).  Write its decomposition into simple factors in \(\KRep{\OO{2n}}\) as \(q_{2n}^*(V_{\vec\gamma}) = \sum_{\vec\mu} n_{\vec\mu} \hat V_{\vec \mu} + n'_{\vec \mu}\hat V_{\vec \mu}\otimes\delta\) for certain non-negative integers \(n_{\vec\mu}, n'_{\vec\mu}\).
As the restrictions of \(V_{\vec\gamma}\) and \(q_{2n}^*(V_{\vec\gamma})\) to \(T\) agree, the above lemma implies the following equality in \(\KRep{T}\):
\[
\begin{aligned}
e^{\vec \gamma} + \left(\ctext{3cm}{a lin.\ comb.\ of \(e^{\vec\mu}\) with \(\vec\mu\prec_{2n+1}\vec\gamma\)}\right)
&= \sum_{\vec \mu}\left[ (n_{\vec\mu} + n'_{\vec\mu})e^{\vec\mu}
+ \left(\ctext{5cm}{a lin.\ comb.\ of \(e^{\vec\nu}\) with\\ \(\vec\nu\) \((2n+1)\)-dominant and\\ \(\vec\nu\prec_{2n+1}\vec\mu\)}\right)\right]\\
&\phantom{=} {\vrule height 5ex width 0pt} + \left(\ctext{5cm}{a lin.\ comb.\ of \(e^{\vec\nu}\) with \(\vec\nu\) non-\((2n+1)\)-dominant}\right)
\end{aligned}
\]

It follows that all \(\vec\mu\) that appear with non-zero coefficient \(n_{\vec\mu} + n'_{\vec\mu}\) on the right-hand side satisfy \(\vec\mu\preceq_{2n+1}\vec\gamma\), that \(\vec\gamma\) occurs among these \(\vec\mu\), and that \(n_{\vec\gamma} + n'_{\vec\gamma} = 1\). Thus:
    \begin{equation}\label{eq:O2n-K-decomp}
      q_{2n}^*(V_{\vec\gamma})
      = \left\{\begin{aligned}
          &\hat V_{\vec\gamma}\text{ or} \\
          &\hat V_{\vec \gamma}\otimes\delta
        \end{aligned}\right\}
      + \sum_{\vec\mu\prec_{2n+1} \vec\gamma} (n_{\vec\mu} \hat V_{\vec\mu} + n'_{\vec\mu}\hat V_{\vec\mu} \otimes\delta) \\
      \quad \text{ in } \KRep{\OO{2n}}
    \end{equation}
    This implies that \(1\) and \(\delta\) generate \(\KRep{\OO{2n}}\) as an \(\im(q_{2n}^*)\)-module:
    \begin{quote}
      Suppose the set of \((2n+1)\)-dominant weights for which \(\hat V_{\vec\gamma}\)  is not contained in \(\im(q_{2n}^*)\vecspan{1,\delta}\) is non-empty.  Let \(\vec\gamma\) be a minimal element with respect to the \((2n+1)\)-ordering.
     Then neither \(\hat V_{\vec\gamma}\) nor \(\hat V_{\vec\gamma}\otimes\delta\) will be contained in \(\im(q_{2n}^*)\vecspan{1,\delta}\). But one of these will appear in the decomposition of \(q_{2n}^*(V_{\vec\gamma})\)  as the \emph{only} summand not contained in \(\im(q_{2n}^*)\vecspan{1,\delta}\), which is impossible.
    \end{quote}

    The same argument applies to \(\GW\).
    Indeed, let \(\nu_{\vec\gamma}\) be an arbitrary symmetry on the simple \(\SO{2n+1}\)-module \(V_{\vec\gamma}\).
    Consider the restriction of  \((V_{\vec\gamma},\nu_{\vec\gamma})\) to \(\OO{2n}\). This restriction decomposes into a sum of simple symmetric objects in \(\GWRep{\OO{2n}}\) whose underlying objects are determined by the decomposition \eqref{eq:O2n-K-decomp} in \(\KRep{\OO{2n}}\) (see \Cref{symmetric-JH-decomp}):
    \begin{equation*}
      q_{2n}^*(V_{\vec\gamma},\nu_{\vec\gamma})
      = \left\{\begin{aligned}
          &(\hat V_{\vec\gamma},\theta_{\vec\gamma}) \text{ or} \\
          &(\hat V_{\vec \gamma},\theta_{\vec\gamma})\otimes\delta^+
        \end{aligned}\right\}
      + \sum_{\vec\mu\prec_{2n+1} \vec\gamma} (\phi_{\vec\mu} (\hat V_{\vec\mu}, \theta_{\vec\mu})
      + \phi'_{\vec\mu} (\hat V_{\vec\mu},\theta'_{\vec\mu})\otimes\delta^+ )
      \quad \text{ in } \GWRep{\OO{2n}}
    \end{equation*}
    for certain symmetries \(\nu_{\vec\gamma}\), \(\theta_{\vec \mu}\) and \(\theta'_{\vec \mu}\) and certain coefficients \(\phi_{\vec \mu},\phi'_{\vec \mu}\in\GW(F)\).
    A simple symmetric \(\OO{2n}\)-module \((\hat V_{\vec\gamma},\theta_{\vec\gamma})\) is contained in the \(\GW(F)\)-submodule
    \(\im(q_{2n}^*)\vecspan{1^+,\delta^+}\) for \emph{some} symmetric form \(\theta_{\vec\gamma}\) on \(\hat V_{\vec\gamma}\)
    if and only if it is contained therein for \emph{every} such form, as the possible symmetries on \(\hat V_{\vec\gamma}\) differ only by invertible scalars.
    We can therefore conclude as above, by considering a weight \(\vec\gamma\)  minimal among all those \((2n+1)\)-dominant weights for which \((\hat V_{\vec\gamma},\theta_{\vec\gamma})\) is not contained in \(\im(q_{2n}^*)\vecspan{1^+,\delta^+}\).

  \item Finally, for the product \(\OO{m_1}\times \OO{m_2}\) we consider the morphism of pre-\(\lambda\)-rings
    \[
    \GWRep{\SO{2n_1+1}\times\SO{2n_2+1}} \xrightarrow{(q_{m_1}\times q_{m_2})^*} \GWRep{\OO{m_1}\times\OO{m_2}}
    \]
    where \(n_i = \lfloor m_i/2 \rfloor\).
    The simple \(\OO{2m_1}\times\OO{2m_2}\)-modules are of the form \(W_1\otimes W_2\), where \(W_1\) and \(W_2\) are irreducible representations of \(\OO{m_1}\) and \(\OO{m_2}\). Let \(\delta_1^+\) and \(\delta_2^+\) be the one-dimensional modules on which \(\OO{m_1}\) and \(\OO{m_2}\) act via the determinant, each equipped with some symmetry.
Let \(\omega_1\) and \(\omega_2\) be symmetries on \(W_1\) and \(W_2\).
As we have seen in the two previous parts of the proof, we can write
    \[
\begin{aligned}
      (W_1,\omega_1) &= q_{m_1}^* (x_1) + \delta_1^+q_{m_1}^*(x_2)\\
      (W_2,\omega_2) &= q_{m_2}^* (y_1) + \delta_2^+q_{m_2}^*(y_2)
    \end{aligned}
\]
    for certain elements \(x_i,y_i\in \GWRep{\SO{2n_i+1}}\).
    In \(\GWRep{\OO{m_1}\times\OO{m_2}}\) we therefore have
    \[
\begin{aligned}
      (W_1\otimes W_2, \omega_1\otimes\omega_2) = \sum_{i,j \in \{ 0, 1\} }(\delta^+_1)^i(\delta^+_2)^j(q_{m_1}\times q_{m_2})^*(x_i\otimes y_j).
    \end{aligned}
\]
    Thus, \(\GWRep{\OO{m_1}\times\OO{m_2}}\) is generated over \(\im\left((q_{m_1}\times q _{m_2})^*\right)\) by the line elements \(1^+, \delta_1^+\) and \(\delta_2^+\), and we may once more conclude via \cref{lem:special-via-generators}.
  \end{asparaenum}
\end{proof}

\section{Appendix: Representations of product group schemes}\label{sec:appendix}
In this appendix, we collect some basic facts from the representation theory of affine algebraic group schemes.
Though we assume that all of these are well-known, we have been unable to locate precise references.
Recall that all our representations are finite-dimensional.  The ground field is denoted \(F\), as elsewhere in this article, but no assumption on the characteristic is necessary in this appendix.
\subsection{Representations of direct products}
Let \(G_1\) and \(G_2\) be affine algebraic group schemes over a field \(F\).
\begin{mdprop}\label{reps-of-products}
  If \(E_1\) and \(E_2\) are simple \(G_1\)- and \(G_2\)-modules, respectively, and if \(\End_{G_2}(E_2) = F\), then
  \[
  E_1\otimes E_2
  \]
is a simple \(G_1\times G_2\)-module.  Conversely, if \(\End_{G_2}(E) = F\) for every simple \(G_2\)-module \(E\), then every simple \(G_1\times G_2\)-module is of the above form.
\end{mdprop}
Over algebraically closed fields, the conditions on endomorphisms become vacuous.  In this case, a proof may be found in \cite{Steinberg:Lectures}*{\S\,12}.
The corresponding statement for Lie algebras, over any field, is included in \cite{BlockZassenhaus}*{\S\,3}.
\begin{proof}
  For the first statement, let \(W\) be any non-zero \(G_1\times G_2\)-submodule of \(E_1\otimes E_2\).  Then, as \(G_2\)-modules, \(E_1\otimes E_2\) and \(W\) are both finite direct sums of copies of \(E_2\).  Moreover, by our assumption on endomorphisms, the inclusion \(W\hookrightarrow E_1\otimes E_2\) may be described by a matrix with coefficients in \(F\).
This implies that, as a \(G_2\)-module, \(W = V\otimes E_2\) for some non-zero vector subspace \(V\) of \(E_1\).
In fact, \(V\) is necessarily a \(G_1\)-submodule of \(E_1\):  any non-zero linear form \(\phi\) on \(E_2\) induces a morphism of \(G_1\)-modules \(\id\otimes\phi\colon E_1\otimes E_2\to E_1\), and the composition \[V\otimes E_2\to E_1\otimes E_2 \to E_1\] is then a morphism of \(G_1\)-modules with image \(V\). As \(E_1\) is simple, \(V\) must be equal to \(E_1\) and, thus, \(W\) must be equal to \(E_1\otimes E_2\).

  For the second statement, let \(M\) be a simple \(G_1\times G_2\)-module.
  Then for any simple \(G_2\)-module \(E\)
  we have an isomorphism of \(G_1\times G_2\)-modules
  \[
  \Hom_{G_2}(E,M)\otimes E\cong (\mathrm{soc}_{G_2} M)_E
  \]
  \cite{Jantzen}*{6.15~(2)}.
  On the left-hand side, the action of \(G_1\times G_2\) on \(\Hom_{G_2}(E,M)\) is induced by its usual action on \(\Hom(E,M)\); it follows that \(G_2\) acts trivially on \(\Hom_{G_2}(E,M)\) and that we may view \(\Hom_{G_2}(E,M)\) as a \(G_1\)-representation.
  The action of \(G_1\times G_2\) on \(E\) is obtained by trivially extending the given action of \(G_2\) on E.
  On the right-hand side, \(\mathrm{soc}_{G_2}M\) denotes the \(G_2\)-socle of \(M\) and \((-)_E\) denotes its \(E\)-isotypical part.
  The claim is, in particular, that \((\mathrm{soc}_{G_2} M)_E\) is a \(G_1\times G_2\)-submodule of \(M\).

  There must be some simple \(G_2\)-module \(E\) for which \((\mathrm{soc}_{G_2}M)_E\) is non-zero \cite{Jantzen}*{2.14~(2)}.
  As \(M\) itself is simple, we find
  \[
  \Hom_{G_2}(E,M)\otimes E\cong M
  \]
  for this \(E\).
  It remains to note that \(\Hom_{G_2}(E,M)\) is simple as a \(G_1\)-module:
  if it contained a proper \(G_1\)-submodule, \(M\) would contain a proper \(G_1\times G_2\)-submodule.
\end{proof}

\subsection{Representations of semi-direct products}
We are mainly interested in semi-direct products with \(\Z/2\): in the orthogonal group \(\OO{2n}\cong \SO{2n}\rtimes\Z/2\) and in the ``extended tori'' \(T\rtimes\Z/2\).  In slightly greater generality, we describe in this section how to obtain the simple representations of a semi-direct product \(H\rtimes\Z/p\) (\(p\) a prime) from the simple representations of \(H\).  All corresponding facts for representations of compact Lie groups may be found in \cite{BtD:Lie}*{VI.7}. The only exception is \Cref{reps-of-sdp:semi-simple}, which we learnt from Skip Garibaldi; see \cite{BGL:LinearPreservers}*{Proposition~2.2}.

Recall that a representation of an affine algebraic group scheme \(G\) consists of an \(F\)-vector space \(V\) together with a natural \(A\)-linear action of \(G(A)\) on \(V\otimes A\) for every \(F\)-algebra \(A\).  Such a representation is completely determined by its restriction to \emph{connected} \(F\)-algebras, so we may assume that \(A\) is connected whenever this is convenient.
In particular, as \((\Z/p)(A) = \Z/p\) for any connected \(A\), a representation of the constant group scheme \(\Z/p\) is the same as a representation of the corresponding abstract group.  An \(H\rtimes \Z/p\)-module is an \(H\)-module which is also a \(\Z/p\)-module, such that
\[
  (x.h).(x.v) = x.(h.v) \quad\in V\otimes A
\]
for all connected \(F\)-algebras \(A\) and all \(x\in \Z/p\), \(h\in H(A)\) and \(v\in V\).
\begin{defn}\label{def:twisted-rep}
  Let \(V\) be a representation of some normal subgroup scheme \(H\subset G\).
  Given an element \(g\in G(F)\), the \define{\(g\)-twisted \(H\)-module \(\twist{g}V \)} has the same underlying \(F\)-vector space \(V\), with \(h\in H(A)\) acting on \(\twist{g}V \otimes A\) as \(g h g^{-1}\) does on \(V\otimes A\).
\end{defn}

\begin{lem}\label{reps-twist}
Let \(V\) be an \(H\rtimes\Z/p\)-module.
If \(U\subset V\) is an \(H\)-submodule, then \(x.U\subset V\) is also an \(H\)-submodule, for any \(x\in\Z/p\), and the action of \(x\) on \(V\) induces an isomorphism of \(H\)-modules
\[
\twist{-x}U\xrightarrow[x]{\cong} x.U.\qedhere
\]
\end{lem}

\begin{mdprop}\label{reps-of-sdp:main}
Let \(V\) be a simple \(H\rtimes\Z/p\)-module. Then
\[
\begin{aligned}
  &\text{either}\quad && V \text{ is simple as an \(H\)-module} \\
  &\text{or}     && V\cong \oplus_{x\in\Z/p} \twist{x}U \quad \lttext{7cm}{for some simple \(H\)-module \(U\) that does not lift to an \(H\rtimes\Z/p\)-module.}
\end{aligned}
\]
In the second case, the action of \(y\in\Z/p\) on \((u_x)_{x}\in\oplus_x\twist{x}U\) is given by \(y.(u_x)_x = (u_{x+y})_x\).
\end{mdprop}

\begin{proof}
By \cite{Jantzen}*{2.14~(2)}, we can find a simple \(H\)-submodule \(U\subset V\). Then
\(
\sum_x x.U
\)
is also an \(H\)-submodule of \(V\),
and in fact an \(H\rtimes\Z/p\)-submodule.
As \(V\) is simple, it follows that
  \[
  \sum_x x.U = V.
  \]
The stabilizer of \(U\) for the action of \(\Z/p\) on the set of \(F\)-sub-vector spaces \(\{x.U\}\) is either trivial or the whole group, so either all \(x.U\) are equal as \(F\)-sub-vector spaces or they are all distinct. In the first case, we have \(V_{|H} = U\), so \(V_{|H}\) is a simple \(H\)-module. In the second case, the \(H\)-submodules \(x.U\) must intersect trivially, so the above sum is direct and we have an isomorphism of \(H\rtimes\Z/p\)-modules
\[
\begin{aligned}
\bigoplus_x \twist{x}U &\xrightarrow{\cong} \bigoplus_x x.U = V  \\
(u_x)_x &\mapsto (x.u_{-x})_x,
\end{aligned}
\]
with \(\Z/p\) acting on the left-hand side as described in the \namecref{reps-of-sdp:main}.
Finally, suppose that the \(H\)-module structure on \(U\) was the restriction of some \(H\rtimes\Z/p\)-module structure. Then we could define a monomorphism of \(H\rtimes\Z/p\)-modules
\[
\begin{aligned}
  U &\hookrightarrow \bigoplus_x \twist{x} U\\
  u &\mapsto (x.u)_x,
\end{aligned}
\]
contradicting the simplicity of \(V\).
\end{proof}

It remains to determine which simple \(H\)-modules do lift to \(H\rtimes\Z/p\)-modules, and in how many ways.

\begin{mdlem}\label{reps-of-sdp:lifts}
A lift of an \(H\)-module structure on an \(F\)-vector space \(U\) to an \(H\rtimes\Z/p\)-module structure corresponds to an isomorphism of \(H\)-modules
\[
\phi\colon U\xrightarrow{\cong}\twist{1}U
\]
such that \(\phi^p\colon U\xrightarrow\cong\twist{1}U\xrightarrow\cong\twist{2}U\xrightarrow{\cong}\cdots\xrightarrow{\cong}\twist{p}U = U\) is the identity.
\end{mdlem}

\begin{proof}
Given an \(H\rtimes\Z/p\)-module structure on \(U\) lifting the given \(H\)-module structure, take \(\phi\) to be the action of \(1\in\Z/p\). Conversely, given an isomorphism \(\phi\) as described, define the \(\Z/p\)-action on \(u\in U\) by \(x.u := \phi^x.u\).
\end{proof}

\begin{rem}
Suppose \(U\) is a simple \(H\)-module \emph{over an algebraically closed field of characteristic not \(p\)}. Then
\(U\) can be lifted to an \(H\rtimes\Z/p\)-module if and only if \(U\cong U^1\) through an arbitrary isomorphism.
Indeed, as \(F\) is algebraically closed, \(\mathrm{End}_F(H) = F\), so \(\phi^p\) is a scalar. Moreover, \(F\) contains primitive \(p^{\text{th}}\) roots,  so we can normalize \(\phi\) such that \(\phi^p=1\).
\end{rem}

\begin{mdlem}\label{reps-of-sdp:different-lifts}
  If an \(H\)-module \(U\) can be lifted to an  \(H\rtimes\Z/p\)-module \(\hat U\), then
  \[
  \hat U\otimes e^\lambda
  \]
  is also a lift for any character \(e^\lambda\) of \(\Z/p\).
  If, moreover, \(\End_H(U) = F\), all lifts of \(U\) are of this form.
\end{mdlem}

\begin{proof}
  The first claim is clear since \(e^\lambda_{|H}\) is the trivial representation.
  For the second statement, suppose \(\phi\) and \(\psi\) are two different isomorphisms \(U\cong\twist{1}U\) such that \(\phi^p = \psi^p  = \id\). Then \(\psi^{-1}\circ\phi\) is an automorphism of the \(H\)-module \(U\), hence under the assumption \(\End_H(U)= F\) it is given by multiplication with some scalar \(\lambda\in F\).  Moreover, \(\lambda^p = 1\).  Writing \(U^\phi\) and \(U^\psi\) for the lifts of \(U\) determined by \(\phi\) and \(\psi\), and writing \(e^\lambda\) for the character of \(\Z/p\) defined by \(\lambda\), we have an isomorphism of \(H\rtimes\Z/p\)-modules
\begin{align*}
    U^{\psi}\otimes e^\lambda&\xrightarrow{\cong}U^\phi\\
   u&\mapsto \lambda^{-1} u\qedhere
\end{align*}
\end{proof}

\begin{rem}
  For \(H\rtimes\Z/p\)-modules of the form \(V=\bigoplus_{x\in\Z/p}\twist{x}U\), we have isomorphisms of \(H\rtimes\Z/p\)-modules
  \[
  V\otimes e^\lambda \xrightarrow{\cong} V
  \]
  given by \((u_x)_x \mapsto  (\lambda^x u_x)_x\), for any character \(e^\lambda\) of \(\Z/p\).
\end{rem}

When \(H\) is split semi-simple, the preceding lemmas can be simplified as follows.
Fix a maximal split torus and a Borel subgroup \(T\subset B\subset H\).  Let \(T^*:= \Hom(T,\Gm)\) be the character lattice, and let \(\Delta\subset T^*\) be the set of simple roots corresponding to our choice of \(B\).
Any automorphism \(\alpha\) of \(H\) determines a unique automorphism \(\alpha^*\) of \(T^*\) that restricts to a permutation of \(\Delta\).  Indeed, for any such \(\alpha\), there exists an element \(h\in H(F)\) such that the composition \(\alpha'\) of \(\alpha\) with conjugation by \(h\) stabilizes \(T\) and \(B\) \cite{Borel:LAG}*{19.2, 20.9(i)}. This automorphism \(\alpha'\) is well-defined up to conjugation by an element of \(T(F)\) and thus induces a well-defined automorphism \(\alpha^*\) of \(T^*\).  Of course, \(\alpha^*\) is completely determined by its restriction to \(\Delta\), which is, in fact, an automorphism of the Dynkin diagram of \(H\).

\begin{cor}\label{reps-of-sdp:semi-simple}
Let \(U\) be an irreducible representation of a split semi-simple algebraic group scheme \(H\) over \(F\).
  Suppose that the generator \(1\in \Z/p\) acts on \(H\) via an automorphism \(\alpha\).
\begin{compactitem}
\item[(i)] The representation \(U\) can be lifted to a representation \(\hat U\) of \(H\rtimes\Z/p\) if and only if the induced automorphism \(\alpha^*\) of \(T^*\) fixes the highest weight of \(U\).
\end{compactitem}
Moreover, if a lift \(\hat U\) exists, then:
\begin{compactitem}
\item[(ii)] All other lifts are of the form \(\hat U \otimes e^{\lambda}\), where \(\lambda\) is a character of \(\Z/p\).
\item[(iii)] The lifts \(\hat U \otimes e^{\lambda}\) for different characters \(\lambda\) are all distinct.
\end{compactitem}
\end{cor}
\begin{proof}
  \begin{asparaitem}
  \item[\it (i)] Let \(\vec\mu\) be the highest weight of \(U\).
  We claim that \(\twist{1}U\) is the simple \(H\)-module with highest weight \(\alpha^*(\vec\mu)\).
  Indeed, suppose first that \(\alpha\) can be restricted to both \(B\) and \(T\), \ie that \(\alpha=\alpha'\) in the notation above.  Then for any character \(\vec\nu\in T^*\), \(\twist{1}e^{\vec\nu} = e^{\alpha^*{\vec\nu}}\).
  Moreover, as \(\alpha^*\) sends simple roots to simple roots, it preserves the induced partial order on \(T^*\).
  So the claim follows.
  In general, the same argument applies to \(\twist{1'}U\), the twist of \(U\) by \(1\in\Z/p\) acting via \(\alpha'\).
  As \(\alpha\) and \(\alpha'\) differ only by an inner automorphism of \(H\), this representation is isomorphic to \(\twist{1}U\).

  We thus find that an isomorphism \(\phi\colon U \to \twist{1}U\) exists if and only if \(\alpha^*\) fixes \(\vec\mu\).  It remains to show that such an isomorphism will always satisfy the condition named in \Cref{reps-of-sdp:lifts}.  To see this, we consider the inclusion
  \[
  \Hom_H(U,\twist{1}U) \hookrightarrow \Hom_T(U, \twist{1}U).
  \]
  When \(U\) and \(\twist{1}U\) are isomorphic, both contain \(e^{\vec\mu}\) with multiplicity one, so the right-hand side contains \(\End_T(e^{\vec\mu}) = F\) as a direct factor. If we choose our \(H\)-equivariant isomorphism \(\phi\) such that \(\phi\) restricts to the identity on \(e^{\vec\mu}\), then \(\phi^p\) will also restrict to the identity on \(e^{\vec\mu}\).  On the other hand, \(\phi^p\) will be multiplication with a scalar in any case, so for such choice of \(\phi\) it will be the identity.
  \item[\it (ii)] follows directly from \Cref{reps-of-sdp:different-lifts}.
  \item[\it (iii)] The proof of {\it (i)} shows that we may choose \(\hat U\) such that the restriction of the \(\Z/p\)-action to \(e^{\vec\mu}\subset \hat U\) is trivial. Then the restriction of the \(\Z/p\)-action on \(\hat U\otimes e^{\lambda}\) to \(e^{\vec\mu}\) is given by \(\lambda\).  So the lifts \(\hat U\otimes e^{\lambda}\) for different \(\lambda\) are non-isomorphic.\qedhere
  \end{asparaitem}
\end{proof}

In the final lemma, we write \(\dimend_G(E)\) for the dimension of \(\End_G(E)\) over \(F\).

\begin{mdlem}[Endomorphism rings]\label{reps-of-sdp:endo}
Let \(H\) be an arbitrary affine algebraic group scheme over \(F\), and let \(V\) be a simple \(H\rtimes\Z/p\)-module. Then
\[
\begin{aligned}
\dimend_{H\rtimes\Z/p}(V) &\leq \dimend_H(V) &&\text{ if \(V_{|H}\) is simple}\\
\dimend_{H\rtimes\Z/p}(V) &= \dimend_H(U) && \text{ if \(V_{|H}\cong\textstyle\bigoplus_x\twist{x}U\) such that \(U\not\cong \twist{1}U\)}\\
\dimend_{H\rtimes\Z/p}(V) &= p\cdot\dimend_H(U) && \text{ if \(V_{|H}\cong\textstyle\bigoplus_x\twist{x}U\) such that \(U \cong \twist{1}U\)}
\end{aligned}
\]
Note that the isomorphism \(U\cong\twist{1}U\) in the last statement necessarily violates the condition of \Cref{reps-of-sdp:lifts}.
\end{mdlem}

\begin{proof}
The first statement is clear since \(\End_{H\rtimes\Z/p}(V)\subset \End_H(V)\).
For the other two cases, note that \(V\cong \ind_H^{H\rtimes\Z/p}(U)\).
Indeed, for any \(H\)-module \(U\) we have an isomorphism of \(H\rtimes\Z/p\)-modules
\[
\begin{aligned}
   \ind_H^{H\rtimes\Z/p}(U) \xrightarrow{\cong} \mathrm{Mor}(\Z/p, U_a) \cong \bigoplus_x\twist{x}U
\end{aligned}
\]
\cite{Jantzen}*{3.8~(3)}, with \(y\in\Z/p\) acting on the right-hand side by \(y.(u_x)_x = (u_{x+y})_x\).
 Explicitly, this isomorphism can be described as follows:
The representation \(\ind U\) can be identified with a subspace of the vector space of  natural transformations from \(H\rtimes\Z/p\) to \(U_a\), with the action of \(g\in (H\rtimes\Z/p)(A)\) on such a transformation \(F\) given by \((g.F) := F\circ g^{-1}\) \cite{Jantzen}*{3.3~(2)}. We send \(F\) to \((F(1,-x))_x\in\bigoplus_x\twist{x}U\).

Given the identification of \(V\) with \(\ind U\), the \namecref{reps-of-sdp:endo} follows from a straight-forward calculation:
\[
\begin{aligned}
   \dimend_{H\rtimes\Z/p}(V)
   &= \dimhom_{H\rtimes\Z/p}(\ind U, \ind U)
   = \dimhom_H((\ind U)_{|H}, U) \\
   &= \dimhom_H(\textstyle\bigoplus_x\twist{x}U, U)
   = \sum_x \dimhom_H(\twist{x}U, U)
\end{aligned}
\]
If the \(\twist{x}U\) are non-isomorphic, then only the summand \(\dimend_H(U)\) corresponding to the trivial twist is non-zero.
Otherwise, all \(p\) summands are equal to \(\dimend_H(U)\).
\end{proof}

\subsection*{Acknowledgements}
My interest in \(\lambda\)-structures on Grothendieck-Witt rings stems from conversations with Pierre Guillot.  I thank Lars Hesselholt for drawing my attention to Borger's philosophy on \(\lambda\)-rings, Jo\"el Riou and Skip Garibaldi for suggesting several improvements in the appendix, and  Jim Humphreys for some references.

\end{document}